\documentclass[12pt]{amsart}       
\usepackage{txfonts}
\usepackage{amssymb}
\usepackage{eucal}
\usepackage{graphicx}
\usepackage{amsmath}
\usepackage{amscd}
\usepackage[all]{xy}           
\usepackage{tikz}
\usepackage{amsfonts,latexsym}
\usepackage{xspace}
\usepackage{epsfig}
\usepackage{float}
\usepackage{color}
\usepackage{fancybox}
\usepackage{colordvi}
\usepackage{multicol}
\usepackage{colordvi}
\usepackage{longtable}
\usepackage[colorlinks,final,backref=page,hyperindex,hypertex]{hyperref}
\usepackage[active]{srcltx} 

\newtheorem{theorem}{Theorem}[section]
\newtheorem{prop}[theorem]{Proposition}
\theoremstyle{definition}
\newtheorem{defn}[theorem]{Definition}
\newtheorem{lemma}[theorem]{Lemma}

\newtheorem{prop-def}{Proposition-Definition}[section]
\newtheorem{coro-def}{Corollary-Definition}[section]

\newtheorem{exam}{Example}[section]

\newcommand{\nc}{\newcommand}
\nc{\tred}[1]{\textcolor{red}{#1}}
\nc{\tblue}[1]{\textcolor{blue}{#1}}
\nc{\tgreen}[1]{\textcolor{green}{#1}}
\nc{\tpurple}[1]{\textcolor{purple}{#1}}
\nc{\btred}[1]{\textcolor{red}{\bf #1}}
\nc{\btblue}[1]{\textcolor{blue}{\bf #1}}
\nc{\btgreen}[1]{\textcolor{green}{\bf #1}}
\nc{\btpurple}[1]{\textcolor{purple}{\bf #1}}
\nc{\NN}{{\mathbb N}}
\nc{\ncsha}{{\mbox{\cyr X}^{\mathrm NC}}} \nc{\ncshao}{{\mbox{\cyr
X}^{\mathrm NC}_0}}

\renewcommand{\frak}{\mathfrak}

\newcommand{\efootnote}[1]{}

\renewcommand{\textbf}[1]{}

\newcommand{\delete}[1]{}

\nc{\mlabel}[1]{\label{#1}}  
\nc{\mcite}[1]{\cite{#1}}  
\nc{\mref}[1]{\ref{#1}}  
\nc{\mbibitem}[1]{\bibitem{#1}} 

\delete{
\nc{\mlabel}[1]{\label{#1}  
{\hfill \hspace{1cm}{\small\tt{{\ }\hfill(#1)}}}}
\nc{\mcite}[1]{\cite{#1}{\small{\tt{{\ }(#1)}}}}  
\nc{\mref}[1]{\ref{#1}{{\tt{{\ }(#1)}}}}  
\nc{\mbibitem}[1]{\bibitem[\bf #1]{#1}} 
}

\nc{\mo}{\calm}
\nc{\barot}{\overline{\otimes}}
\nc{\dm}{\diamond_\mo}
\nc{\oot}{\overline{\ot}}
\nc{\opa}{\ast} \nc{\opb}{\odot} \nc{\op}{\bullet} \nc{\pa}{\frakL}
\nc{\arr}{\rightarrow} \nc{\lu}[1]{(#1)} \nc{\mult}{\mrm{mult}}
\nc{\diff}{\mathfrak{Diff}}
\nc{\opc}{\sharp}\nc{\opd}{\natural}
\nc{\ope}{\circ}
\nc{\dpt}{\mathrm{d}}
\nc{\diam}{alternating\xspace}
\nc{\Diam}{Alternating\xspace}
\nc{\cdiam}{canonical alternating\xspace}
\nc{\Cdiam}{Canonical alternating\xspace}
\nc{\AW}{\mathcal{A}}
\nc{\mrbo}{modified RBO\xspace }
\nc{\ari}{\mathrm{ar}}
\nc{\lef}{\mathrm{lef}}
\nc{\Sh}{\mathrm{ST}}
\nc{\Cr}{\mathrm{Cr}}
\nc{\st}{{Schr\"oder tree}\xspace}
\nc{\sts}{{Schr\"oder trees}\xspace}
\nc{\vertset}{\Omega} 
\nc{\pb}{{\mathrm{pb}}}
\nc{\Lf}{{\mathrm{Lf}}}
\nc{\lft}{{left tree}\xspace}
\nc{\lfts}{{left trees}\xspace}
\nc{\fat}{{fundamental averaging tree}\xspace}
\nc{\fats}{{fundamental averaging trees}\xspace}
\nc{\avt}{\mathrm{Avt}}
\nc{\rass}{{\mathit{RAss}}}
\nc{\aass}{{\mathit{AAss}}}
\nc{\vin}{{\mathrm Vin}}    
\nc{\lin}{{\mathrm Lin}}    
\nc{\inv}{\mathrm{I}n}
\nc{\gensp}{V} 
\nc{\genbas}{\mathcal{V}} 
\nc{\bvp}{V_P}     
\nc{\gop}{{\,\omega\,}}     
\nc{\bin}[2]{ (_{\stackrel{\scs{#1}}{\scs{#2}}})}  
\nc{\binc}[2]{ \left (\!\! \begin{array}{c} \scs{#1}\\
    \scs{#2} \end{array}\!\! \right )}  
\nc{\bincc}[2]{  \left ( {\scs{#1} \atop
    \vspace{-1cm}\scs{#2}} \right )}  
\nc{\bs}{\bar{S}} \nc{\cosum}{\sqsubset} \nc{\la}{\longrightarrow}
\nc{\rar}{\rightarrow} \nc{\dar}{\downarrow} \nc{\dprod}{**}
\nc{\dap}[1]{\downarrow \rlap{$\scriptstyle{#1}$}}
\nc{\md}{\mathrm{dth}} \nc{\uap}[1]{\uparrow
\rlap{$\scriptstyle{#1}$}} \nc{\defeq}{\stackrel{\rm def}{=}}
\nc{\disp}[1]{\displaystyle{#1}} \nc{\dotcup}{\
\displaystyle{\bigcup^\bullet}\ } \nc{\gzeta}{\bar{\zeta}}
\nc{\hcm}{\ \hat{,}\ } \nc{\hts}{\hat{\otimes}}
\nc{\free}[1]{\bar{#1}}
\nc{\uni}[1]{\tilde{#1}} \nc{\hcirc}{\hat{\circ}} \nc{\lleft}{[}
\nc{\lright}{]} \nc{\lc}{\lfloor} \nc{\rc}{\rfloor}
\nc{\curlyl}{\left \{ \begin{array}{c} {} \\ {} \end{array}
    \right .  \!\!\!\!\!\!\!}
\nc{\curlyr}{ \!\!\!\!\!\!\!
    \left . \begin{array}{c} {} \\ {} \end{array}
    \right \} }
\nc{\longmid}{\left | \begin{array}{c} {} \\ {} \end{array}
    \right . \!\!\!\!\!\!\!}
\nc{\onetree}{\bullet} \nc{\ora}[1]{\stackrel{#1}{\rar}}
\nc{\ola}[1]{\stackrel{#1}{\la}}
\nc{\ot}{\otimes} \nc{\mot}{{{\boxtimes\,}}}
\nc{\otm}{\overline{\boxtimes}} \nc{\sprod}{\bullet}
\nc{\scs}[1]{\scriptstyle{#1}} \nc{\mrm}[1]{{\rm #1}}
\nc{\margin}[1]{\marginpar{\rm #1}}   
\nc{\dirlim}{\displaystyle{\lim_{\longrightarrow}}\,}
\nc{\invlim}{\displaystyle{\lim_{\longleftarrow}}\,}
\nc{\mvp}{\vspace{0.3cm}} \nc{\tk}{^{(k)}} \nc{\tp}{^\prime}
\nc{\ttp}{^{\prime\prime}} \nc{\svp}{\vspace{2cm}}
\nc{\vp}{\vspace{8cm}} \nc{\proofbegin}{\noindent{\bf Proof: }}
\nc{\proofend}{$\blacksquare$ \vspace{0.3cm}}
\nc{\modg}[1]{\!<\!\!{#1}\!\!>}
\nc{\intg}[1]{F_C(#1)} \nc{\lmodg}{\!
<\!\!} \nc{\rmodg}{\!\!>\!}
\nc{\cpi}{\widehat{\Pi}}
\nc{\sha}{{\mbox{\cyr X}}}  
\nc{\ssha}{{\mbox{\cyrs X}}} 
\nc{\shpr}{\diamond}    
\nc{\shp}{\ast} \nc{\shplus}{\shpr^+}
\nc{\shprc}{\shpr_c}    
\nc{\msh}{\ast} \nc{\zprod}{m_0} \nc{\oprod}{m_1}
\nc{\vep}{\varepsilon} \nc{\labs}{\mid\!} \nc{\rabs}{\!\mid}
\nc{\sqmon}[1]{\langle #1\rangle}

\nc{\mmbox}[1]{\mbox{\ #1\ }} \nc{\dep}{\mrm{dep}} \nc{\fp}{\mrm{FP}}
\nc{\rchar}{\mrm{char}} \nc{\End}{\mrm{End}} \nc{\Fil}{\mrm{Fil}}
\nc{\Mor}{Mor\xspace} \nc{\gmzvs}{gMZV\xspace}
\nc{\gmzv}{gMZV\xspace} \nc{\mzv}{MZV\xspace}
\nc{\mzvs}{MZVs\xspace} \nc{\Hom}{\mrm{Hom}} \nc{\id}{\mrm{id}}
\nc{\im}{\mrm{im}} \nc{\incl}{\mrm{incl}} \nc{\map}{\mrm{Map}}
\nc{\mchar}{\rm char} \nc{\nz}{\rm NZ} \nc{\supp}{\mathrm Supp}

\nc{\Alg}{\mathbf{Alg}} \nc{\Bax}{\mathbf{Bax}} \nc{\bff}{\mathbf f}
\nc{\bfk}{{\bf k}} \nc{\bfone}{{\bf 1}} \nc{\bfx}{\mathbf x}
\nc{\bfy}{\mathbf y}
\nc{\base}[1]{\bfone^{\otimes ({#1}+1)}} 
\nc{\Cat}{\mathbf{Cat}}

\nc{\detail}{\marginpar{\bf More detail}
    \noindent{\bf Need more detail!}
    \svp}
\nc{\Int}{\mathbf{Int}} \nc{\Mon}{\mathbf{Mon}}
\nc{\rbtm}{{shuffle }} \nc{\rbto}{{Rota-Baxter }}
\nc{\remarks}{\noindent{\bf Remarks: }} \nc{\Rings}{\mathbf{Rings}}
\nc{\Sets}{\mathbf{Sets}} \nc{\wtot}{\widetilde{\odot}}
\nc{\wast}{\widetilde{\ast}} \nc{\bodot}{\bar{\odot}}
\nc{\bast}{\bar{\ast}} \nc{\hodot}[1]{\odot^{#1}}
\nc{\hast}[1]{\ast^{#1}} \nc{\mal}{\mathcal{O}}
\nc{\tet}{\tilde{\ast}} \nc{\teot}{\tilde{\odot}}
\nc{\oex}{\overline{x}} \nc{\oey}{\overline{y}}
\nc{\oez}{\overline{z}} \nc{\oef}{\overline{f}}
\nc{\oea}{\overline{a}} \nc{\oeb}{\overline{b}}
\nc{\weast}[1]{\widetilde{\ast}^{#1}}
\nc{\weodot}[1]{\widetilde{\odot}^{#1}} \nc{\hstar}[1]{\star^{#1}}
\nc{\lae}{\langle} \nc{\rae}{\rangle}
\nc{\lf}{\lfloor}
\nc{\rf}{\rfloor}

\nc{\QQ}{{\mathbb Q}}
\nc{\RR}{{\mathbb R}} \nc{\ZZ}{{\mathbb Z}}
\nc{\CC}{{\mathbb C}}

\nc{\cala}{{\mathcal A}} \nc{\calb}{{\mathcal B}}
\nc{\calc}{{\mathcal C}}
\nc{\cald}{{\mathcal D}} \nc{\cale}{{\mathcal E}}
\nc{\calf}{{\mathcal F}} \nc{\calg}{{\mathcal G}}
\nc{\calh}{{\mathcal H}} \nc{\cali}{{\mathcal I}}
\nc{\call}{{\mathcal L}} \nc{\calm}{{\mathcal M}}
\nc{\caln}{{\mathcal N}} \nc{\calo}{{\mathcal O}}
\nc{\calp}{{\mathcal P}} \nc{\calr}{{\mathcal R}}
\nc{\cals}{{\mathcal S}} \nc{\calt}{{\mathcal T}}
\nc{\calu}{{\mathcal U}} \nc{\calw}{{\mathcal W}} \nc{\calk}{{\mathcal K}}
\nc{\calx}{{\mathcal X}} \nc{\CA}{\mathcal{A}}

\nc{\fraka}{{\mathfrak a}} \nc{\frakA}{{\mathfrak A}}
\nc{\frakb}{{\mathfrak b}} \nc{\frakB}{{\mathfrak B}}
\nc{\frakc}{{\mathfrak c}}
\nc{\frakD}{{\mathfrak D}} \nc{\frakF}{\mathfrak{F}}
\nc{\frakf}{{\mathfrak f}} \nc{\frakg}{{\mathfrak g}}
\nc{\frakJ}{\mathfrak{J}}
\nc{\frakH}{{\mathfrak H}} \nc{\frakL}{{\mathfrak L}}
\nc{\frakM}{{\mathfrak M}} \nc{\bfrakM}{\overline{\frakM}}
\nc{\frakm}{{\mathfrak m}} \nc{\frakP}{{\mathfrak P}}
\nc{\frakN}{{\mathfrak N}} \nc{\frakp}{{\mathfrak p}}
\nc{\frakS}{{\mathfrak S}} \nc{\frakT}{\mathfrak{T}}

\nc{\BS}{\mathbb{S
}}

\font\cyr=wncyr10 \font\cyrs=wncyr7
\nc{\li}[1]{\textcolor{red}{Li:#1}}
\nc{\xigou}[1]{\textcolor{blue}{Xigou: #1}}
\nc{\xing}[1]{\textcolor{purple}{Xing:#1}}
\nc{\UN}{U_{N}}
\nc{\FN}{F_{\mathrm M}}
\nc{\altx}{\Lambda}
\nc{\spr}{\cdot}
\nc{\rts}{\stackrel{\rightarrow}{\shpr}}
\nc{\ox}{\overline{\frak x}}
\nc{\oX}{\overline{X}}
\nc{\epro}{\diamond} \nc{\epr}{\overline{\diamond}_\mo}
\nc{\cum}{{\textstyle \int}}

\begin{document}

\title{Commutative modified Rota-Baxter algebras, shuffle products and Hopf algebras}

\author{Xigou Zhang}
\address{Department of Mathematics, Jiangxi Normal University, Nanchang, Jiangxi 330022, China}
\email{xyzhang@jxnu.edu.cn}

\author{Xing Gao}
\address{School of Mathematics and Statistics, Key Laboratory of Applied Mathematics and Complex Systems, Lanzhou University, Lanzhou, 730000, P.R. China}
\email{gaoxing@lzu.edu.cn}

\author{Li Guo}
\address{Department of Mathematics and Computer Science,
         Rutgers University,
         Newark, NJ 07102, USA}
\email{liguo@rutgers.edu}

\date{\today}
\begin{abstract}
In this paper, we begin a systematic study of modified Rota-Baxter algebras, as an associative analogue of the modified classical Yang-Baxter equation. We construct free commutative modified Rota-Baxter algebras by a variation of the shuffle product and describe the structure both recursively and explicitly. We then provide these algebras with a Hopf algebra structure by applying a Hochschild cocycle.
\end{abstract}

\keywords{Modified Rota-Baxter algebra, Rota-Baxter algebra, shuffle product, bialgebra, Hopf algebra}

\subjclass[2010]{16T99,16W99,16S10}

\maketitle

\tableofcontents

\setcounter{section}{0}

\allowdisplaybreaks

\section{Introduction}

A {\bf Rota-Baxter operator} of weight $\lambda$ (where $\lambda$ is a constant) is defined to be a linear operator $P$ on an associative algebra $R$ satisfying
$$ P(x)P(y)=P(P(x)y) + P(xP(y)) +\lambda P(xy) \  \text{for all}\  x, y\in R.$$
Then $(R, P)$ is called a {\bf Rota-Baxter algebra}. The study of Rota-Baxter algebras originated from the work~\mcite{Ba} of G. Baxter on fluctuation theory of probability in 1960. It was studied by well-known mathematicians such as Atkinson, Cartier and Rota~\mcite{FV,Ca,Ro} in the 1960-70s. Its study has experienced a quite remarkable renascence in the recent decades with many applications in mathematics and physics~\mcite{Ag,BBGN,EG,GGZ,GK1,GK3,JZ,ML,MY}, most notably the work of Connes and Kreimer on renormalization of quantum field theory~\mcite{CK,EGK,EGM}. See~\mcite{Gub} for further details and references.

Back in the 1980s, Semonov-Tian-Shansky made the discovery that, under suitable conditions, the Rota-Baxter identity on a Lie algebra is precisely the operator form of the classical Yang-Baxter equation, named after the well-known physicists. Also introduced in that paper is the closely related {\it modified} classical Yang-Baxter equation~\mcite{BGN1,Bo,KS,Sem}:

$$ [P(x),P(y)]=P[P(x),y]+P[x,P(y)]-xy,$$
later found applications in the study of generalized Lax pairs and affine geometry on Lie groups. A modified Rota-Baxter algebra~\mcite{E} is defined to be an associative algebra with a linear operator which satisfies the associative analogue of the modified classical Yang-Baxter equation and has since been applied to the study of extended $\calo$-operators, associative Yang-Baxter equations, infinitesimal bialgebras and dendriform algebras~\mcite{BGN2,BGN3}.

A modified classical Yang-Baxter equation can be obtained from the classical Yang-Baxter equation by a linear transformation, but plays independent role in the physics study. Thus from an application point of view, it is worthwhile to study the algebraic structure of the modified classical Yang-Baxter equations and its associative analogues, the modified Rota-Baxter operators. We begin a systematic study of modified Rota-Baxter algebras, emphasizing the commutative case. We consider modified Rota-Baxter algebras of any weight $\kappa$, with the classical modified Rota-Baxter algebra being the case when $\kappa$ is a negative square. As we will notice later in the paper, the structure of modified Rota-Baxter algebras differs significantly from that of Rota-Baxter algebras. From theoretical point of view, an algebraic or combinatorial object is often studied by multiple structures it possesses. For example, the transformation matrices of the space of symmetric functions with respect to its various bases are essential in the study of algebraic combinatorics. We also hope that a systematic study of modified Rota-Baxter algebras will shed new light in understanding Rota-Baxter algebras.

The layout of the paper is as follows. Section~\mref{sec:bas} gives some basic properties of modified Rota-Baxter algebras. Section~\mref{sec:free} provides the construction of free commutative modified Rota-Baxter algebras on another commutative algebra. Section~\mref{sec:hopf} equips these free commutative modified Rota-Baxter algebras with a Hopf algebra structure, when the weights are negative squares and when the base algebra is a connected Hopf algebra. The method is to apply a suitable cocycle property to give a recursion.
\smallskip

\noindent
{\bf Notations.} Throughout this paper, an algebra is taken to be over a commutative unitary algebra $\bfk$, as are the linear maps and tensor products.

\section{General properties of modified Rota-Baxter algebras}
\mlabel{sec:bas}
We give the general definition of modified Rota-Baxter algebras.

\begin{defn}
Let $R$ be a $\bfk$-algebra and $\kappa\in \bfk$. A linear map $P:R\to R$ is called a {\bf modified Rota-Baxter operator} of weight $\kappa$ if $P$ satisfies the operator identity
\begin{equation}
P(u)P(v)=P(uP(v))+P(P(u)v)+\kappa uv  \  \text{for all } u, v\in R.
\mlabel{eq:mrbo}
\end{equation}
\end{defn}
Then the pair $(R, P)$ or simply $R$ is called a {\bf modified Rota-Baxter algebra} of weight  $\kappa$.
The class of modified Rota-Baxter algebras of weight $\kappa$ forms a category, with the morphisms being algebra homomorphisms between the algebras that commute the linear operators.

As observed in~\mcite{E}, there is an interesting relation between a Rota-Baxter algebra  and  a modified Rota-Baxter algebra.

\begin{lemma}
Let $(R, P)$ be a Rota-Baxter algebra of weight  $\lambda$. Define $Q:=-\lambda\, \id -2P $. Then $(R, Q)$ is a modified Rota-Baxter algebra of weight  $-\lambda^2$. In particular, a Rota-Baxter algebra of weight zero is a modified Rota-Baxter algebra of weight zero.
\mlabel{lem:ptoq}
 \end{lemma}

Thus only modified Rota-Baxter algebras whose weight are negative squares correspond to Rota-Baxter algebras. From this relation, the following examples of modified Rota-Baxter algebras are immediate.

\begin{exam} (Integration)
Let $R$ be the $\RR$-algebra of continuous functions on $\RR$. Define $P:R\to R$ by
the integration $$P(f)(x) = \cum_0^x f(t) dt. $$
Then $P$ is a Rota-Baxter operator of weight 0~\cite[Example 1.1.4]{Gub} and so a modified
Rota-Baxter of weight 0 by Eq.~(\mref{eq:mrbo}).
\end{exam}

\begin{exam} (Scalar product)
Let $R$ be a \bfk-algebra. For any given $\lambda\in \bfk$, the operator
$$P_\lambda: R \to R\, \quad r\mapsto -\lambda r$$
is a Rota-Baxter operator of weight $\lambda$~\cite[Exercise 1.1.7]{Gub}. So the operator
$$Q_\lambda: = -\lambda\id - 2P : R \to R\, \quad r\mapsto \lambda r$$
is a modified Rota-Baxter operator of weight $-\lambda^2$.
\end{exam}

By observation, a linear operator $P$ satisfies Eq.~(\mref{eq:mrbo}) if and only if $-P$ satisfies the same equation. Thus we have
\begin{prop}
Let $R$ be a \bfk-algebra and $P:R\to R$ a linear operator.
Then $P$ is a modified Rota-Baxter operator on $R$ if and only if $-P$ is one.
\mlabel{pp:pmp}
\end{prop}

Furthermore we characterize involutional modified Rota-Baxter operators as follows.

\begin{theorem}
Suppose that 2 is invertible in $\bfk$. Let $R$ be a \bfk-algebra and $P:R\to R$ be a linear operator. Then the following statements are equivalent.
\begin{enumerate}
\item The operator $P$ is an involutional (i.e. $P^2 = \id$)
modified Rota-Baxter operator of weight -1 on $R$; \mlabel{it:pmrb}

\item There is a \bfk-module direct sum
decomposition $R = R_1 \oplus R_2$ of $R$ into nonunitary \bfk-subalgebras $R_1$ and $R_2$
of $R$ such that
$$P:R\to R, \quad u_1 + u_2 \mapsto u_1 - u_2 $$
for $u_1\in R_1$ and $u_2\in R_2$. \mlabel{it:idecom}
\end{enumerate}
\mlabel{imrbo}
\end{theorem}

As an examples of Theorem~\mref{imrbo} is $R=\CC[\vep^{-1},\vep]]$, the algebra of Laurent series for which we take $R_1:=\CC[[\vep]]$ and $R_2:=\vep^{-1}\CC[\vep^{-1}]]$.

\begin{proof}
((\ref{it:idecom})$\Longrightarrow$ (\ref{it:pmrb})) Let $u=u_1+u_2$ and $v = v_1 + v_2$ be in $R$ with $u_1, v_1\in R_1$ and
$u_2, v_2\in R_2$. Then $u_1v_1\in R_1$ and $u_2v_2\in R_2$. On the one hand, we have
$$P(u)P(v) = (u_1 - u_2) (v_1 -v_2) = u_1v_1 - u_1v_2 - u_2v_1 + u_2v_2.$$
On the other hand,
\begin{align*}
P(uP(v)) =& P((u_1 +u_2) (v_1 - v_2)) = P(u_1v_1 - u_1v_2 + u_2v_1 - u_2v_2)\\
=& u_1v_1 - P(u_1v_2) + P(u_2v_1) + u_2v_2,\\
P(P(u)v) =& P((u_1 - u_2)(v_1+v_2)) = P(u_1v_1 + u_1v_2 - u_2v_1 - u_2v_2)\\
=& u_1v_1 + P(u_1v_2) - P(u_2v_1) + u_2v_2,\\
uv =& (u_1 + u_2)(v_1+v_2) = u_1v_1 + u_1v_2 + u_2v_1 + u_2v_2.
\end{align*}
In summary, we have $$ P(u)P(v) = P(uP(v))+P(P(u)v) - uv,$$
whence $P$ is a modified Rota-Baxter operator of weight -1 on $R$. Furthermore,
$$P^2(u) = P(u_1-u_2) = u_1 + u_2 = u$$
and so $P^2 = \id$.

((\ref{it:pmrb}) $\Longrightarrow$ (\ref{it:idecom}))
Write $$R_1:= (\id + P)(R)\,\text{ and }\, R_2:= (\id - P)(R).$$
Since $P$ is \bfk-linear, $R_1$ and $R_2$ are submodules of $R$. To prove $R_1$
is a subalgebra of $R$, let $u+P(u)$ and $v+P(v)$ be in $R_1$. Then
\begin{align*}
(u+P(u))(v+P(v)) =& uv + uP(v) + P(u)v + P(u)P(v) \\
=& uv + uP(v) + P(u)v +P(uP(v)) + P(P(u)v) - uv\\
=& (\id+P)(uP(v)) + (\id+P)(P(u)v) \in R_1.
\end{align*}
Similarly, $R_2$ is a subalgebra of $R$.
Since $P$ is involutive and 2 is invertible in $\bfk$, we have

$$R = R_1\oplus R_2,$$
with $R_1$ and $R_2$ being the eigen-submodules of eigenvalues 1 and -1 respectively.
Thus for any $u = u_1 + u_2\in R$ with $u_1\in R_1$ and $u_2\in R_2$,  we have $P(u)=u_1-u_2$.
\end{proof}

Modified Rota-Baxter algebras of various weights can be related to one another as follows.

\begin{prop}
Let $R$ be a \bfk-algebra.
\begin{enumerate}
\item If $P$ is a modified Rota-Baxter operator of weight 1 on $R$, then $\kappa P$
is a modified Rota-Baxter operator of weight $\kappa^2$ on $R$; \mlabel{it:w1tok1}

\item If $P$ is a modified Rota-Baxter operator of weight $\kappa^2$ on $R$ and $\kappa$
is invertible in $\bfk$, then $\kappa^{-1} P$
is a modified Rota-Baxter operator of weight 1 on $R$. \mlabel{it:w1tok2}
\end{enumerate}
\mlabel{pp:w1tok}
\end{prop}

\begin{proof}
(\mref{it:w1tok1}) From
\begin{equation}
P(u)P(v)=P(uP(v))+P(P(u)v)+  uv  \  \text{for all } u, v\in R, \mlabel{eq:mrb1}
\end{equation}
we obtain
\begin{equation}
(\kappa P)(u)(\kappa P)(v)=(\kappa P)(u (\kappa P)(v))+ (\kappa P)((\kappa P)(u)v) + \kappa^2 uv  \  \text{for all } u, v\in R,  \mlabel{eq:mrb2}
\end{equation}
whence $\kappa P$ a modified Rota-Baxter operator of weight $\kappa^2$.

(\mref{it:w1tok2}) If $\kappa$ is invertible, then Eq.~(\mref{eq:mrb2}) also implies
Eq.~(\mref{eq:mrb1}).
\end{proof}

\section{Free commutative modified Rota-Baxter algebras of weight $\kappa$}
\mlabel{sec:free}

We construct the free objects in the category of commutative modified Rota-Baxter algebras. The module structure is the same as the one for free commutative Rota-Baxter algebras as constructed in~\mcite{GK1,GK2}. We define the multiplication of the free objects first by an recursion in Section~\mref{ss:gen}. Then a non-recursive formula is given in Section~\mref{ss:spmo} by a variation of the stuffle product. We end the section by a special case.

\subsection{The general construction}
\mlabel{ss:gen}

We begin with the definition.

\begin{defn}
Let $A$ be a given commutative $\bfk$-algebra. A {\bf free commutative modified Rota-Baxter algebra} on $A$ is a commutative modified Rota-Baxter $\bfk$-algebra $F_{MRB}(A)$ together with an algebra homomorphism $j_A: A\to F_{MRB}(A)$ that satisfies the following universal property: for any free commutative modified Rota-Baxter $k$-algebra $(R,P)$ of weight $\kappa$ and algebra homomorphism $f: A\to R$, there is a unique modified Rota-Baxter $k$-algebra homomorphism $\free {f}:(F_{MRB}(A), P_\mo)\to (R, P)$ such that $f=\free {f}\circ j_A$.
\mlabel{de:freemrba}
\end{defn}

We first give the underlying module of the free commutative modified Rota-Baxter algebra over $A$.
Denote
\begin{equation}
\sha_\mo (A):= \bigoplus_{k\geq 1} A^{\ot k}= A \oplus (A\ot A) \oplus \cdots.
\mlabel{eq:modrb}
\end{equation}

We next equip this $\bfk$-module with a multiplication $\diamond_\mo$, called the {\bf modified quasi-shuffle product}.\footnote{As is well-known, the quasi-shuffle product coincides with the stuffle product~\mcite{Gub}. We will show in Section~\mref{ss:spmo}, the modified quasi-shuffle product can also be interpreted by a modified stuffle product.} It is given by a recursion.
For pure tensors $\fraka= a_0 \ot  a_1 \ot \cdots \ot  a_m\in A^{\ot (m+1)}, \frakb= b_0 \ot  b_1 \ot \cdots \ot  b_n\in A^{\ot (n+1)}, m, n\geq 0,$ write $\fraka=a_0\ot \fraka'$ with $\fraka'\in A^{\ot m}$ if $m\geq 1$ and $\frakb=b_0\ot \frakb'$ with $\frakb'\in A^{\ot n}$ if $n\geq 1$. The recursion is given on the sum $m+n\geq 0$ by
\begin{equation}
\fraka\diamond_\mo \frakb:=\left\{\begin{array}{ll}
 a_0  b_0, & \text{for  }m=0,n=0;\\
 a_0  b_0 \ot  a_1  \ot \cdots \ot  a_m, &
\text{for  }m\geq 1,n=0;\\
 a_0  b_0 \ot  b_1  \ot \cdots \ot  b_n, & \text{for  }m=0,n\geq 1;\\
 a_0  b_0 \ot \Big((1 \ot \fraka')\diamond_\mo  \frakb'\Big) + a_0  b_0 \ot \Big(\fraka'\diamond_\mo (1 \ot \frakb')\Big) & \\
\quad +\kappa a_0   b_0 (\fraka'\diamond_\mo \frakb'), & \text{for  }m\geq 1,n\geq 1.\end{array}\right .
\mlabel{eq:mrbprod}
\end{equation}
Here in the last case, for each pair of pure tensors involved in the multiplication $\diamond_\mo$ in the three terms, either one of the pure tensors has length one or the sums of the lengths of the pair of pure tensors is less then $m+n$. Hence the recursion terminates and gives a well-defined binary operation on $\sha_\mo(A)$ after extending by biadditivity.

\begin{theorem}Let a commutative algebra $A$ and $\kappa\in \bfk$ be given.
\begin{enumerate}
\item The pair $(\sha_\mo(A),\diamond_\mo)$ is a commutative $\bfk$-algebra.
    \mlabel{it:alg}
\item Define
    \begin{equation}
    P_A : \sha_\mo(A)\to \sha_\mo(A), \quad \fraka \mapsto 1_A\ot \fraka \quad \text{ for all } \fraka\in \sha_\mo(A).
    \mlabel{eq:op}
    \end{equation}
Then the triple $(\sha_\mo(A), \diamond_\mo,P_A )$ together with the embedding $j_A:A\to \sha_\mo(A)$ is the free commutative modified Rota-Baxter $k$-algebra of weight $\kappa$.
\mlabel{it:comfree}
\end{enumerate}
\mlabel{thm:comfree}
\end{theorem}

\begin{proof}
(\mref{it:alg} )
The commutativity of $\diamond_\mo$ can be seen from the symmetry of the two arguments in the definition of $\diamond_\mo $ in Eq.~(\mref{eq:mrbprod}).

To prove the associativity of $\diamond_\mo$, we only need to verify
\begin{equation}
(\fraka \diamond_\mo \frakb)\diamond_\mo \frakc = \fraka \diamond_\mo (\frakb \diamond_\mo \frakc)
\mlabel{eq:asso}
\end{equation}
for pure tensors $\fraka\in A^{\ot (m+1)}, \frakb\in A^{\ot (n+1)}, \frakc\in A^{\ot (\ell+1)}, m, n, \ell\geq 0$. For this purpose we apply the induction on the sum $m+n+\ell$, which is at least zero.
The initial case of $m+n+\ell=0$ follows from the first case of the definition of $\diamond_\mo$ in Eq.~(\mref{eq:mrbprod}) and the associativity of the multiplication of $A$. Assume that Eq.~(\mref{eq:asso}) has been verified when $m+n+\ell=k$ for a given $k\geq 0$ and consider the case when $m+n+\ell=k+1$.

If at least one of $m, n$ or $\ell$ is zero, then from the second and third cases in the definition of $\diamond_\mo $ in Eq.~(\mref{eq:mrbprod}), we easily obtain
$$(\fraka\diamond_\mo \frakb)\diamond_\mo c=\fraka\diamond_\mo (\frakb\diamond_\mo c).$$

Now for $m\geq 1, n\geq 1, \ell\geq 1$, denote $\fraka= a_0 \ot \fraka', \frakb= b_0 \ot \frakb', c=c_0 \ot \frakc',$ where $a_0, b_0, c_0\in A, \fraka'\in A^{\ot m}, \frakb'\in A^{\ot n}, \frakc'\in A^{\ot \ell}$. For the left hand side of Eq.~(\mref{eq:asso}), we have
 \begin{eqnarray*}
 &&(\fraka\diamond_\mo \frakb)\diamond_\mo c\\
 &=&(( a_0 \ot \fraka')\diamond_\mo ( b_0 \ot \frakb'))\diamond_\mo (c_0 \ot\frakc')\\
 &=&\Big( a_0  b_0 \ot \big((1 \ot \fraka')\diamond_\mo  \frakb'\big) + a_0  b_0 \ot \big(\fraka'\diamond_\mo (1 \ot \frakb')\big) +\kappa  a_0   b_0 (\fraka'\diamond_\mo \frakb')\Big)\diamond_\mo (c_0 \ot\frakc')\\
 &=&\Big( a_0  b_0 \ot \big((1 \ot \fraka')\diamond_\mo  \frakb'\big)\Big)\diamond_\mo (c_0 \ot\frakc')+\Big( a_0  b_0 \ot \big(\fraka'\diamond_\mo (1 \ot \frakb')\big)\Big)\diamond_\mo (c_0 \ot\frakc')\\
 &&+\Big(\kappa  a_0   b_0 (\fraka'\diamond_\mo \frakb')\Big)\diamond_\mo (c_0 \ot\frakc') \\
  &=&( a_0  b_0 c_0) \ot \Big( 1\ot \big((1 \ot \fraka')\diamond_\mo  \frakb'\big)\diamond_\mo\frakc'\Big)
  +( a_0  b_0 c_0) \ot \Big(  \big((1 \ot \fraka')\diamond_\mo  \frakb'\big)\diamond_\mo (1\ot\frakc')\Big) \\
 &&+\kappa ( a_0  b_0 c_0)  \Big(\big((1 \ot \fraka')\diamond_\mo  \frakb'\big)\diamond_\mo\frakc'\Big)
+( a_0  b_0 c_0) \ot \Big( 1\ot \big( \fraka'\diamond_\mo ( 1\ot \frakb')\big)\diamond_\mo\frakc'\Big)\\
&&+( a_0  b_0 c_0) \ot \Big(  \big( \fraka'\diamond_\mo (1\ot  \frakb')\big)\diamond_\mo (1\ot\frakc')\Big)+\kappa ( a_0  b_0 c_0)  \Big(\big( \fraka'\diamond_\mo (1\ot  \frakb')\big)\diamond_\mo\frakc'\Big)\\
&&+\kappa ( a_0   b_0 c_0) \Big((\fraka'\diamond_\mo \frakb')\diamond_\mo (1 \ot\frakc')\Big).
\end{eqnarray*}
For the fifth term on the right hand side of the last equation, by the induction hypothesis we obtain
\begin{eqnarray*}
\lefteqn{ \big( \fraka'\diamond_\mo (1\ot  \frakb')\big)\diamond_\mo (1\ot\frakc')
= \fraka'\diamond_\mo \big((1\ot  \frakb')\diamond_\mo (1\ot\frakc')\big)}\\
&=&  \fraka'\diamond_\mo \Big(1\ot \big((1\ot  \frakb')\diamond_\mo\frakc'\big) \Big)
+ \fraka'\diamond_\mo \Big(1 \ot \big(\frakb'\diamond_\mo (1 \ot\frakc')\big)\Big) + \kappa\fraka'\diamond_\mo \big((\frakb'\diamond_\mo\frakc')\big).
\end{eqnarray*}
Therefore, the left hand side of Eq.~(\mref{eq:asso}) expands to
\begin{eqnarray*}
&&(\fraka\diamond_\mo \frakb)\diamond_\mo c\\
&=&( a_0  b_0 c_0) \ot \Big( 1\ot \big((1 \ot \fraka')\diamond_\mo  \frakb'\big)\diamond_\mo\frakc'\Big)+( a_0  b_0 c_0) \ot \Big(  \big((1 \ot \fraka')\diamond_\mo  \frakb'\big)\diamond_\mo (1\ot\frakc')\Big)\\
&&+\kappa ( a_0  b_0 c_0) \Big( \big((1 \ot \fraka')\diamond_\mo  \frakb'\big)\diamond_\mo\frakc'\Big)
+( a_0  b_0 c_0) \ot \Big( 1\ot \big( \fraka'\diamond_\mo ( 1\ot \frakb')\big)\diamond_\mo\frakc'\Big)\\
&&+( a_0  b_0 c_0) \ot \Big( \fraka'\diamond_\mo \big(1\ot \big((1\ot  \frakb')\diamond_\mo\frakc'\big)\big)\Big)
+( a_0  b_0 c_0) \ot \Big(  \fraka'\diamond_\mo \big(1 \ot (\frakb'\diamond_\mo (1 \ot\frakc'))\big)\Big) \\
&&+\kappa( a_0  b_0 c_0) \ot   \big( \fraka'\diamond_\mo (\frakb'\diamond_\mo\frakc')\big)
+\kappa ( a_0  b_0 c_0)  \Big(\big( \fraka'\diamond_\mo (1\ot  \frakb')\big)\diamond_\mo\frakc'\Big)\\
&&
+\kappa ( a_0   b_0 c_0) \Big((\fraka'\diamond_\mo \frakb')\diamond_\mo (1 \ot\frakc')\Big).
\end{eqnarray*}
For the right hand side of Eq.~(\mref{eq:asso}), we have
\begin{eqnarray*}
\lefteqn{ \fraka\diamond_\mo (\frakb\diamond_\mo c)=( a_0 \ot \fraka')\diamond_\mo \big(( b_0 \ot \frakb')\diamond_\mo (c_0 \ot\frakc')\big) }\\
&=&( a_0 \ot  \fraka')\diamond_\mo \Big( b_0 c_0 \ot \big((1 \ot \frakb')\diamond_\mo\frakc'\big) + b_0 c_0 \ot \big(\frakb'\diamond_\mo (1 \ot\frakc')\big) +\kappa  b_0  c_0 (\frakb'\diamond_\mo\frakc')\Big) \\
&=&( a_0 \ot  \fraka')\diamond_\mo \Big( b_0 c_0 \ot \big((1 \ot \frakb')\diamond_\mo\frakc'\big)\Big)
+( a_0 \ot  \fraka')\diamond_\mo \Big( b_0 c_0 \ot \big(\frakb'\diamond_\mo (1 \ot\frakc')\big)\Big)\\
&&+( a_0 \ot  \fraka')\diamond_\mo \Big(\kappa  b_0  c_0 (\frakb'\diamond_\mo \frakb')\Big)\\
&=&( a_0  b_0 c_0) \ot \Big((1 \ot \fraka')\diamond_\mo \big((1\ot  \frakb')\diamond_\mo\frakc'\big)\Big)+( a_0  b_0 c_0) \ot \Big(  \fraka'\diamond_\mo \big(1\ot \big((1\ot \frakb')\diamond_\mo\frakc'\big)\big)\Big)\\
&&+\kappa ( a_0  b_0 c_0)  \Big( \fraka'\diamond_\mo \big((1\ot \frakb')\diamond_\mo\frakc'\big)\Big)
+( a_0  b_0 c_0) \ot \Big( (1\ot  \fraka')\diamond_\mo \big( \frakb'\diamond_\mo (1\ot\frakc')\big)\Big)\\
&&+( a_0  b_0 c_0) \ot \Big( \fraka'\diamond_\mo \big(1\ot \big( \frakb'\diamond_\mo (1\ot\frakc')\big)\big)\Big)
+\kappa ( a_0  b_0 c_0)  \Big(\fraka'\diamond_\mo \big( \frakb'\diamond_\mo (1\ot\frakc')\big)\Big)\\
&&+\kappa ( a_0  b_0  c_0)\Big((1 \ot \fraka')\diamond_\mo (\frakb'\diamond_\mo\frakc')\Big).
\end{eqnarray*}
For the first term on the right hand side of the last equation, by the inductive hypothesis we obtain
\begin{eqnarray*}
\lefteqn{   (1\ot \fraka')\diamond_\mo \big((1\ot  \frakb')\diamond_\mo\frakc'\big) = \big( (1\ot \fraka')\diamond_\mo (1\ot  \frakb') \big)\diamond_\mo\frakc'}\\
&=& \big( 1\ot \big((1\ot \fraka')\diamond_\mo  \frakb'\big)\big)\diamond_\mo\frakc' +1 \ot \big(\fraka'\diamond_\mo (1 \ot \frakb')\big)\big)\diamond_\mo\frakc' + \kappa(\fraka'\diamond_\mo \frakb')\diamond_\mo\frakc'.
\end{eqnarray*}
Therefore the right hand side of Eq.~(\mref{eq:asso}) expands to \begin{eqnarray*}
&&\fraka\diamond_\mo (\frakb\diamond_\mo c)\\
&=&( a_0  b_0 c_0) \ot \Big(\big( 1\ot \big((1\ot \fraka')\diamond_\mo  \frakb'\big)\big)\diamond_\mo\frakc'\Big)
+( a_0  b_0 c_0) \ot \Big( \big(\fraka'\diamond_\mo (1 \ot \frakb')\big)\diamond_\mo\frakc'\Big) \\
&& + \kappa( a_0  b_0 c_0) \ot \Big((\fraka'\diamond_\mo \frakb')\diamond_\mo\frakc'\Big)
+( a_0  b_0 c_0) \ot \Big(  \fraka'\diamond_\mo \big(1\ot \big((1\ot \frakb')\diamond_\mo\frakc'\big))\Big)\\
&&+\kappa ( a_0  b_0 c_0)  \Big( \fraka'\diamond_\mo \big((1\ot \frakb')\diamond_\mo\frakc'\big)\Big)
+( a_0  b_0 c_0) \ot \Big( (1\ot  \fraka')\diamond_\mo \big( \frakb'\diamond_\mo (1\ot\frakc')\big)\Big)\\
&& +( a_0  b_0 c_0) \ot \Big( \fraka'\diamond_\mo \big(1\ot \big( \frakb'\diamond_\mo (1\ot\frakc')\big)\big)\Big)
+\kappa ( a_0  b_0 c_0)  \Big(\fraka'\diamond_\mo \big( \frakb'\diamond_\mo (1\ot\frakc')\big)\Big)\\
&&+\kappa  (a_0  b_0  c_0)\Big((1 \ot \fraka')\diamond_\mo (\frakb'\diamond_\mo\frakc')\Big).
\end{eqnarray*}
Now we see that the $i$-th term in the expansion of $(\fraka\diamond_\mo \frakb)\diamond_\mo \frakc $ matches with the $\sigma(i)$-th term in the expansion of $\fraka\diamond_\mo (\frakb\diamond_\mo \frakc)$, where $\sigma \in \Sigma_{9}$ is
$$\left (\begin{array}{ll}
\quad i \\
\sigma (i)
\end{array}  \right ) =\left (\begin{array}{ll}
1\quad 2\quad 3\quad 4\quad 5\quad 6\quad 7\quad 8\quad 9  \\
1\quad 6\quad 9\quad 2\quad 4\quad 7\quad 3\quad 5\quad 8
\end{array}  \right ). $$  Thus
$(\fraka\diamond_\mo \frakb)\diamond_\mo \frakc=\fraka\diamond_\mo (\frakb\diamond_\mo \frakc) .$
This completes the inductive proof of the associativity of $\diamond_\mo$.

\smallskip

\noindent
(\mref{it:comfree})
First Eq.~(\ref{eq:mrbprod}) gives
\begin{eqnarray*}
\lefteqn{P_A (\fraka)\diamond_\mo P_A (\frakb)=(1 \ot \fraka)\diamond_\mo (1 \ot \frakb)}\\
&=&1 \ot \Big((1 \ot \fraka)\diamond_\mo  \frakb\Big) + 1 \ot \Big(\fraka\diamond_\mo (1 \ot \frakb)\Big) + \kappa  (\fraka\diamond_\mo \frakb)\\
&=&P_A (P_A (\fraka) \diamond_\mo \frakb) + P_A (\fraka \diamond_\mo P_A (\frakb))+ \kappa  \fraka \diamond_\mo \frakb.
\end{eqnarray*}
Thus $P_A $ is a modified Rota-Baxter operator on  $\sha_\mo (A)$, giving a modified Rota-Baxter algebra $(\sha_\mo(A),P_A)$.

To verify the universal property of the free commutative modified Rota-Baxter algebra, let $(R,P)$ be a modified Rota-Baxter algebra and $f: A \rightarrow R$ be a homomorphism of modified Rota-Baxter algebras. We consider the existence and the uniqueness of the modified Rota-Baxter algebra homomorphism $\free {f}: \sha_\mo^{+} (A) \rightarrow R$ such that $\free{f}\,j_A=f$.

For any pure tensor $\fraka= a_0 \ot  a_1 \ot \cdots \ot  a_k \in A^{\ot (k+1)}, $
we apply the induction on $k$ to define $\free{f}$. When $k=0, A=j_A(A)$ and we define $\free {f}(\fraka)=f(\fraka).$ Note that this the only way to define $\free{f}$ in this case.
Let $n\geq 0$ and assume that $\free {f}(\fraka)$ has been defined for $k\leq n$. Consider $\fraka \in  A^{\ot (n+1)}$. Then $\fraka= a_0 \ot \fraka'$ with $\fraka' \in  A^{\ot n} $ and we have $\fraka= a_0 \diamond_\mo P_A (\fraka')$. Define
$$\free {f}(\fraka):=f( a_0) P(\free {f}(\fraka')),$$
where $\free{f}(\fraka')$ is defined by the inductive hypothesis. Note again that this is the only way to define $\free{f}$ in order for it to be a modified Rota-Baxter algebra homomorphism.

We now show that $ \free {f}$ is a modified Rota-Baxter algebra homomorphism. First,
$$ \free{f}(P_A(\fraka))=\free{f}(1\ot \fraka)=P(\free{f}(\fraka)) \quad \text{for all } \fraka\in A^{\ot (n+1)} $$
by the definition of $\free{f}$. Thus we only need to verify
$$ \free {f}(\fraka\diamond_\mo \frakb)=\free {f}(\fraka)  \free {f}(\frakb) \quad \text{for all } \fraka\in A^{\ot (m+1)}, \frakb\in A^{\ot (n+1)}, m, n\geq 0.$$
We carry out the verification by induction on $m+n\geq 0$. In the initial step of $m+n=0$, we have $\fraka, \frakb\in A$. So
$$\free{f}(\fraka \diamond_\mo \frakb)=\free{f}(\fraka \frakb) = f(\fraka \frakb) =f(\fraka)f(\frakb)=\free{f}(\fraka)\free{f}(\frakb),$$
as needed.

For the inductive step, the verification is still easy if $m=0$ or $n=0$ from the definition of the product $\diamond_\mo$. When $m, n\geq 1$, we have
\begin{eqnarray*}
&&\free {f}(\fraka\diamond_\mo \frakb)=\free {f}(( a_0 \ot \fraka')\diamond_\mo ( b_0 \ot \frakb'))\\
&=&\free {f}(( a_0 \diamond_\mo  P_A (\fraka'))\diamond_\mo ( b_0 \diamond_\mo P_A (\frakb')))\\
&=&\free {f} \Big(( a_0  b_0) \diamond_\mo (P_A (\fraka')\diamond_\mo  P_A (\frakb')) \Big)\\
&=&\free {f} \Big(( a_0  b_0) \diamond_\mo \big(P_A (\fraka'\diamond_\mo  P_A (\frakb'))+P_A (P_A (\fraka')\diamond_\mo \frakb')+ \kappa \fraka' \diamond_\mo \frakb'\big)\Big)\\
&=&\free {f}( a_0  b_0) \Big( (\free {f}P_A )(\fraka'\diamond_\mo  P_A (\frakb'))+(\free {f}P_A )(P_A (\fraka')\diamond_\mo \frakb')+ \kappa \free{f} (\fraka' \diamond_\mo \frakb')\Big)\\
&=&f( a_0  b_0) \Big( (P \free {f})(\fraka'\diamond_\mo  (1 \ot \frakb'))+(P \free {f})((1 \ot \fraka')\diamond_\mo \frakb')+ \kappa \free{f}( \fraka')  \free{f}(\frakb')\Big)\\
&=&f( a_0  b_0) \Big(P\big(\free{f}(\fraka') \free {f}(1 \ot \frakb')\big)+P\big(\free {f}(1 \ot \fraka') \free{f}(\frakb')\big)+ \kappa \free{f}( \fraka')  \free{f}(\frakb')\Big)\\
&=&f( a_0  b_0) \Big(P\big(\free{f}(\fraka') P(\free {f} (\frakb'))\big)+P\big(P(\free {f}(\fraka')) \free{f}(\frakb')\big)+ \kappa \free{f}( \fraka')  \free{f}(\frakb')\Big)\\
&=&f( a_0) f( b_0) P(\free {f}(\fraka')) P(\free {f}( \frakb'))\\
&=&\free {f}(\fraka) \free {f}(\frakb).
\end{eqnarray*}
Therefore $\free{f}$ is a homomorphism of modified Rota-Baxter algebras.

With the uniqueness of $\free{f}$ noted above, the proof of the theorem is completed.
\end{proof}

\subsection{Modified stuffle product and modified quasi-shuffle product}
\mlabel{ss:spmo}

In this subsection, we give an explicit formula for the modified quasi-shuffle product $\diamond_\mo$ defined in Eq.~(\mref{eq:mrbprod}) by a ``modified" version of the stuffle product.

Let $k\geq 1$. Denote $[k]=\{1, \cdots, k\}$.
Let $m\geq 1, n\geq 1$ and $0\leq r\leq \min(m, n)$. Define
\begin{align*}
\frakJ_{m,\,n,\,r} := \left\{ (\varphi, \psi) \,\left| \,
\begin{array}{l}
\varphi:[m] \to [m+n-r], \psi: [n]\to [m+n-r]\\
\text{are order preserving injective maps such that }\\
\im(\varphi)\cup\im(\psi) = [m+n-r]
\end{array}
 \right.
\right\}
\end{align*}
Write
\begin{equation*}
\frakJ_{m,\,n}:= \bigcup_{r=0}^{\min(m,n)} \frakJ_{m,\,n,\,r}.
\end{equation*}
For $(\varphi, \psi)\in \frakJ_{m,\,n,\,r}$, $k\in [m]$ and $\ell\in [n]$, define the {\bf $\varphi$-overlapping degree} of $(\varphi,\psi)$ at $k$ by
$$ d_{\varphi,\,k}:=d_{(\varphi,\,\psi),\,\varphi,\,k}=|\{ i\in [k] \,|\, \varphi(i)\in \im \psi\}|=|\varphi([k])\cap \psi([n])|$$
 and the {\bf $\psi$-overlapping degree} of $(\varphi,\psi)$ at $\ell$ by
$$ d_{\psi,\,\ell}:=d_{(\varphi,\,\psi),\,\psi,\,\ell}=|\{ j\in [\ell] \,|\, \psi(j)\in \im \varphi\}|=|\varphi([m])\cap \psi([\ell])|.$$
We then define
\begin{equation}
\begin{aligned}
\tilde{\varphi}:&[m] \to \{0\}\cup[m+n-2r], \quad \tilde{\varphi}(k):= \varphi(k)-d_{\varphi,k}, \quad k\in [m],\\
\tilde{\psi}:&[n] \to \{0\}\cup[m+n-2r], \quad \tilde{\psi}(\ell):= \psi(\ell)-d_{\psi,\ell}, \quad \ell\in [n].
\end{aligned}
\mlabel{eq:tvp}
\end{equation}
Let $\fraka = a_0 \ot a_1\ot \cdots \ot a_m \in A^{\ot (m+1)}$ and $\frakb = b_0 \ot b_1\ot \cdots \ot b_n\in A^{\ot (n+1)}, m, n\geq 1$.
We define
\begin{equation}
\fraka \epro_{(\varphi,\, \psi)} \frakb:= a_0b_0c_0\ot \cdots \ot c_{m+n-2r},
\quad \text{where }
c_s:=\prod_{i\in \tilde{\varphi}^{-1}(s)}a_i \prod_{j\in\tilde{\psi}^{-1}(s)}b_j,\quad s\in \{0\}\cup [m+n-2r],
\label{eq:exmod}
\end{equation}
with the convention that the product over an empty set is 1.
We give an example before continuing.

\begin{exam}
Take $m=5$, $n=4$ and $r=3$. Consider order preserving injective maps
$\varphi:[5]\to [6]$ and $\psi:[4]\to [6]$ defined by
$$ \begin{tabular}{|c|c|c|c|c|c|}
\hline
$k$ &1&2&3&4&5\\ \hline $\varphi(k)$& 1&2&4&5&6 \\\hline\end{tabular} \quad \text{ and }\quad
\begin{tabular}{|c|c|c|c|c|} \hline $\ell$&1&2&3&4\\ \hline $\psi(\ell)$ &1&3&4&5\\\hline\end{tabular}$$
Then the overlapping degrees and maps $\tilde{\varphi}, \tilde{\psi}$ are given by the following tables, where we reproduce the rows of $\varphi$ and $\psi$ so that each of the fourth rows is simply the difference of the second and the third ones.
$$ \begin{tabular}{|c|c|c|c|c|c|}
\hline
$k$ &1&2&3&4&5\\ \hline
$\varphi(k)$& 1&2&4&5&6 \\ \hline
$d_{\varphi,k}$&1&1&2&3&3\\ \hline
$\tilde{\varphi}(k)$& 0&1&2&2&3 \\\hline\end{tabular} \quad \text{ and }\quad
\begin{tabular}{|c|c|c|c|c|}
\hline $\ell$&1&2&3&4\\ \hline
$\psi(\ell)$ &1&3&4&5\\\hline
$d_{\psi,\ell}$& 1&1&2&3 \\ \hline
$\tilde{\psi}(\ell)$&0&2&2&2\\ \hline
\end{tabular}$$
Thus for $\fraka = a_0\ot a_1 \ot \cdots \ot a_5\in A^{\ot 6}$ and $\frakb = b_0\ot b_1\ot \cdots \ot b_4\in A^{\ot 5}$, Eq.~(\ref{eq:exmod}) gives
$$\fraka \epro_{(\varphi,\, \psi)} \frakb = a_0b_0a_1b_1 \ot a_2\ot b_2a_3b_3a_4b_4 \ot a_5.$$
\end{exam}

Back to the general case, define
\begin{equation}
\fraka \epr \frakb := \left\{\begin{array}{ll}
a_0 \frakb, & m=0, \\ b_0\fraka, & n=0, \\
\sum\limits_{r=0}^{\min(m,\,n)} \kappa^r \left( \sum_{(\varphi,\,\psi)\in \frakJ_{m,\,n,\,r}} \fraka\epro_{(\varphi,\,\psi)} \frakb \right),
& m\geq 1, n\geq 1.
\end{array} \right.
\mlabel{eq:epr}
\end{equation}
Extending by linearity, $\epr$ is defined on $\sha_\mo(A)$.

\begin{exam} When $m=1$ and $n=2$, we have $\fraka=a_0\ot a_1$ and $\frakb = b_0\ot b_1\ot b_2$.
Then $(\varphi, \psi)\in \frakJ_{1,\,2,\,r}, 0\leq r\leq 1,$ and the corresponding $\fraka\epro_{(\varphi,\,\psi)} \frakb$ are given in the following table.
$$\begin{tabular}{|c|c|c|c|c|c|c|}
\hline
r& $m+n-2r$ & m+n-r&$\varphi$ & $\psi$& $\fraka\epro_{(\varphi,\,\psi)} \frakb$\\
\hline
0& 3&3 &$\varphi(1)=1$ & $ \psi(1)=2, \psi(2)=3$ & $a_0b_0\ot a_1\ot b_1\ot b_2$\\
\hline
0&3&3 &$\varphi(1)=2$ & $\psi(1)=1, \psi(2)=3$ & $a_0b_0\ot b_1\ot a_1\ot b_2$\\
\hline
0&3&3 &$\varphi(1)=3$ & $\psi(1)=1, \psi(2)=2$ & $a_0b_0\ot b_1\ot b_2\ot a_1$\\
\hline
1&1&2 &$\varphi(1)=1$ & $\psi(1)=1, \psi(2)=2$ & $\kappa a_0b_0a_1b_1\ot b_2$\\
\hline
1&1&2 &$\varphi(1)=2$ & $\psi(1)=1, \psi(2)=2$ & $\kappa a_0b_0\ot b_1a_1b_2$\\
\hline
\end{tabular}
$$
So we obtain
\begin{align*}
\fraka \epr \frakb =& a_0b_0\ot a_1\ot b_1\ot b_2+  a_0b_0\ot b_1\ot a_1\ot b_2 +
a_0b_0\ot b_1\ot b_2\ot a_1 \\
&+ \kappa a_0b_0a_1b_1\ot b_2 + \kappa a_0b_0\ot b_1a_1b_2.
\end{align*}
\end{exam}

The next result shows that $\epr$ coincides with $\diamond_\mo$.

\begin{prop}
The modified stuffle product $\epr$ coincides with the modified quasi-shuffle product $\diamond_\mo$.
\mlabel{pp:twom}
\end{prop}

\begin{proof}
For simplicity, we take the weight $\kappa$ to be $1$.
We show that $\epr$ satisfies the same boundary conditions and recursion that defines $\diamond_\mo$ in Eq.~(\mref{eq:mrbprod}).
Let $\fraka = a_0\ot a_1\ot \cdots \ot a_m \in A^{\ot (m+1)}$ and $\frakb = b_0 \ot b_1\ot \cdots \ot b_n\in A^{\ot (n+1)}$ with $m,n\geq 0$. If $m=0$ or $n=0$, then by Eqs.~(\mref{eq:mrbprod}) and (\mref{eq:epr}), we have
$$\fraka \epr \frakb =   \fraka\diamond_\mo \frakb.$$
Thus $\epr$ and $\diamond_\mo$ satisfy the same boundary conditions.
If $m\geq 1$ and $n\geq 1$, write
\begin{align*}
\frakJ'_{m,\,n}:=& \{(\varphi,\,\psi) \in \frakJ_{m,\,n} \,|\, \varphi(1) = 1, \psi(1)\neq 1\},\\
\frakJ''_{m,\,n}:=& \{(\varphi,\,\psi) \in \frakJ_{m,\,n} \,|\, \varphi(1) \neq 1, \psi(1) = 1\},\\
\frakJ'''_{m,\,n}:=& \{(\varphi,\,\psi) \in \frakJ_{m,\,n} \,|\, \varphi(1) = 1 = \psi(1)\}.
\end{align*}
Then
$$\frakJ_{m,\,n} = \frakJ'_{m,\,n}\sqcup \frakJ''_{m,\,n}\sqcup \frakJ'''_{m,\,n}.$$
Therefore we get
\begin{align*}
\fraka \epr \frakb =& \sum_{(\varphi,\,\psi)\in \frakJ_{m,\,n}} \fraka \epro_{(\varphi,\,\psi)} \frakb\\
=&\sum_{(\varphi,\,\psi)\in \frakJ'_{m,\,n}} \fraka \epro_{(\varphi,\,\psi)} \frakb + \sum_{(\varphi,\,\psi)\in \frakJ''_{m,\,n}} \fraka \epro_{(\varphi,\,\psi)} \frakb + \sum_{(\varphi,\,\psi)\in \frakJ'''_{m,\,n}} \fraka \epro_{(\varphi,\,\psi)} \frakb\\
=& \sum_{(\varphi',\,\psi)\in \frakJ_{m-1,\,n}} a_0b_0 \ot (\fraka'\epro_{(\varphi',\,\psi)} (1\ot \frakb'))
+ \sum_{(\varphi,\,\psi')\in \frakJ_{m,\,n-1}} a_0b_0 \ot ((1\ot \fraka') \epro_{(\varphi,\,\psi')} \frakb')\\
&+ \sum_{(\varphi',\,\psi')\in \frakJ_{m-1,\,n-1}} a_0b_0 (\fraka' \epro_{(\varphi',\,\psi')} \frakb')\\
=& a_0b_0 \ot (\fraka' \epr (1\ot\frakb')) + a_0b_0 \ot ( (1\ot\fraka') \epr \frakb') + a_0b_0(\fraka' \epr \frakb'),
\end{align*}
where
\begin{align*}
\varphi':=\, \varphi|_{[m]\setminus\{1\}},\, \psi':= \psi|_{[n]\setminus\{1\}},\,\fraka' := a_1\ot \cdots \ot a_m\,\text{ and }\,\frakb' := b_1\ot \cdots \ot b_n.
\end{align*}
Hence $\epr$ and $\diamond_\mo$ satisfy the same recursion. This completes the proof.
\end{proof}

\subsection{The case of $A=\bfk$}
\mlabel{ss:case}
Here we consider a special case of Theorem~\mref{thm:comfree} which provides new examples of modified Rota-Baxter algebras.

Choosing $A=\bfk$ in the construction of free commutative modified Rota-Baxter algebra $\sha_\mo(A)$ in Eq.~(\mref{eq:modrb}),
we have
$$\sha_\mo(\bfk)=\bigoplus_{n\geq 0} {\bfk}^{\ot (n+1)}.$$
Since the tensor product is over $\bfk$, we get $\bfk^{ \ot(n+1)}=\bfk u_n$, where
$$ u_n={\bf 1}^{ \ot(n+1)}=\underbrace{{\bf 1} \ot {\bf 1} \ot \cdots \ot  {\bf 1}}_{\text {(n+1)-factors}}\,\text{ for }\,  n\geq 0.$$
Thus $$\sha_\mo(\bfk)=\bfk \{ u_n\,| \,n\geq 0\}.$$
Then Eq.~(\mref{eq:mrbprod}) gives
\begin{equation*}
\begin{aligned}
 u_0\diamond_\mo  u_n=& u_n,\,  u_1\diamond_\mo  u_1=2  u_2+\kappa  u_0, \\
 u_1\diamond_\mo  u_2=& 3  u_3+2 \kappa  u_1, \, u_2\diamond_\mo  u_2=6 u_4+6\kappa  u_2+\kappa^2  u_0.
 \end{aligned}
\end{equation*}
In general, the multiplication formula of $\sha_\mo(\bfk)$ is given by

\begin{prop} For any $m,n\geq 0$,
$$ u_m\diamond_\mo  u_n=\sum\limits_{r=0}^{\min\{ m, n\}}
{m+n-r\choose m} {m\choose r}\,
\kappa ^r  u_{m+n-2r}.$$
\mlabel{pp:mult}
\end{prop}

\begin{proof}
We proceed by induction on $m+n\geq 0$.
When $m = 0$ or $n = 0$, without loss of generality, let $m = 0$. Then the result follows from  $u_0 \diamond_\mo u_n = u_n$.
When $m,n\geq 1 $, we may assume $m \leq n$ by the commutativity of $\diamond_\mo$. Then
\begin{equation}
\begin{aligned}
 u_m\diamond_\mo  u_n=&{\bf 1} \ot ( u_m\diamond_\mo  u_{n-1})+{\bf 1} \ot ( u_{m-1}\diamond_\mo  u_n)+ \kappa  u_{m-1}\diamond_\mo  u_{n-1}\\
=&\sum\limits_{r=0}^{\min(m,n-1)}
\binc{m+n-1-r}{m} \binc{m}{r}
\kappa ^r  u_{m+n-2r}+\sum\limits_{r=0}^{ m-1}
\binc{m+n-1-r}{m-1} \binc{m-1}{r}\,
\kappa ^r u_{m+n-2r}\\
&+ \sum\limits_{r=0}^{m-1}
\binc{m+n-2-r}{ m-1} \binc{m-1}{ r}\,
\kappa ^{r+1} u_{m+n-2-2r} \quad \quad \text {(by the induction hypothesis).}
\end{aligned}
\label{eq:one1}
\end{equation}
If $m=n$, then $\min(m,n-1) = n-1 = m-1$ and so
\begin{equation}
\sum\limits_{r=0}^{\min(m,\,n-1)}
\binc{m+n-1-r}{m} \binc{m}{r}
\kappa ^r  u_{m+n-2r} = \sum\limits_{r=0}^{m-1}
\binc{m+n-1-r}{m} \binc{m}{r}
\kappa ^r  u_{m+n-2r}
=
\sum\limits_{r=0}^{m}
\binc{m+n-1-r}{m} \binc{m}{r}
\kappa ^r  u_{m+n-2r},
\mlabel{eq:one2}
\end{equation}
where the last step follows from the convention
$$\binc{m+n-1-m}{m} = \binc{n-1}{m} = \binc{m-1}{m} = 0.$$
If $m<n$, then $\min(m,n-1) = m$ and Eq.~(\mref{eq:one2}) holds.

Substitute Eq.~(\mref{eq:one2}) into Eq.~(\mref{eq:one1}) and shift the index in the last summand in Eq.~(\mref{eq:one1}), we get
\begin{align*}
 u_m\diamond_\mo  u_n=&\sum\limits_{r=0}^{m}
\binc{m+n-1-r}{m} \binc{m}{r}
\kappa ^r  u_{m+n-2r}+\sum\limits_{r=0}^{ m-1}
\binc{m+n-1-r}{ m-1} \binc{m-1}{ r}\,
\kappa ^r u_{m+n-2r}\\
&+ \sum\limits_{r=1}^{m}
\binc{m+n-1-r}{ m-1} \binc{m-1}{ r-1}\,
\kappa ^{r} u_{m+n-2r} \\
=&\sum\limits_{r=0}^m
\binc{m+n-1-r}{m} \binc{m}{r}
\kappa ^r  u_{m+n-2r}+\sum\limits_{r=0}^{ m}
\binc{m+n-1-r}{ m-1} \binc{m-1}{ r}\,
\kappa ^r u_{m+n-2r}\\
&+ \sum\limits_{r=0}^{m}
\binc{m+n-1-r}{ m-1} \binc{m-1}{ r-1}\,
\kappa ^{r} u_{m+n-2r} \quad \quad \text {(with the convention $\binc{m-1}{ m} = 0 = \binc{m-1}{ -1}$)}\\
=&\sum\limits_{r=0}^m
\binc{m+n-1-r}{m} \binc{m}{r}
\kappa ^r  u_{m+n-2r}+\sum\limits_{r=0}^{ m}
\binc{m+n-1-r}{ m-1} \left( \binc{m-1}{ r} +  \binc{m-1}{ r-1} \right)\,
\kappa ^r u_{m+n-2r}\\
=&\sum\limits_{r=0}^m
\binc{m+n-1-r}{m} \binc{m}{r}
\kappa ^r  u_{m+n-2r}+\sum\limits_{r=0}^{ m}
\binc{m+n-1-r}{ m-1} \binc{m}{ r}\,
\kappa ^r u_{m+n-2r}\\
=&\sum\limits_{r=0}^m
\left(\binc{m+n-1-r}{m}+\binc{m+n-1-r}{m-1}\right)  \binc{m}{r}
\kappa ^r  u_{m+n-2r}\\
=&\sum\limits_{r=0}^m
\binc{m+n-r}{ m} \binc{m}{ r}\,
\kappa ^r u_{m+n-2r},
\end{align*}
as required.
\end{proof}

\section{The Hopf algebra structure on free commutative modified Rota-Baxter algebras}
\mlabel{sec:hopf}

We now prove that a free modified Rota-Baxter algebra $\sha_\mo(A)$ constructed in Theorem~\mref{thm:comfree} carries a natural Hopf algebra structure, as long as $A$ is already a connected bialgebra (and hence a Hopf algebra) and the weight of the modified Rota-Baxter algebra is a negative square. This applies in particular to free commutative modified Rota-Baxter algebra on a set $X$, which corresponds to the case of $A=\bfk[X]$. We divide the construction into two steps, first obtaining the structure of a bialgebra in Section~\mref{ss:bialg} and then that of a Hopf algebra by a filtration argument in Section~\mref{ss:hopf}. In Section~\mref{ss:case2}, the coproduct is made explicit in the special case when $A=\bfk$.

\subsection{The bialgebra structure}
\mlabel{ss:bialg}

Let $(A,\Delta,\vep,u)$ be a bialgebra.
For distinction, we will use $\barot$ for the tensor symbol for the coproduct $\Delta: A\to A\barot A$.

Fix $\lambda\in \bfk$ and consider the free modified Rota-Baxter algebra $\sha_\mo(A)=\sha_{\mo,-\lambda^2}(A)$ of weight $\kappa=-\lambda^2$ on $A$. To obtain a bialgebra on $\sha_\mo(A)$, we define a coproduct
\begin{equation}
\Delta_\mo: \sha_\mo(A)\to \sha_\mo(A) \oot \sha_\mo(A)
\end{equation}
by making use of the coproduct $\Delta$ on $A$ together with the operator $P_A $, so that a suitable ``cocycle condition" is satisfied.
More precisely, we define $\Delta_\mo(\fraka)$ for $\fraka\in A^{\ot k}, k\geq 1,$ inductively on $k$. When $k=1$, that is, $\fraka\in A$, simply define
 $$\Delta_\mo(\fraka):=\Delta(\fraka).$$
For $k\geq 2$ and $\fraka\in A^{\ot k}$, write $\fraka=a_0\ot \fraka'$ with $\fraka'\in A^{\ot k}$. Note that $\fraka= a_0 \diamond_\mo P_A  (\fraka').$ Then we apply a Hochschild cocycle and define
\begin{equation}
\Delta_\mo(P_A  (\fraka'))
:=P_A  (\fraka')\barot 1+(\id\barot P_A  )\Delta_\mo(\fraka')+ \lambda \fraka'\barot 1,
\mlabel{eq:cocyc}
\end{equation}
and
\begin{equation}
\Delta_\mo(\fraka)=\Delta_\mo( a_0 \diamond_\mo P_A  (\fraka')):=\Delta ( a_0)\diamond_\mo \Delta_\mo(P_A  (\fraka')).
\mlabel{eq:coprod}
\end{equation}
From this definition and the multipicativity of $\Delta$ we obtain
$$ \Delta_\mo(a\dm \fraka)=\Delta(a) \dm \Delta_\mo(\fraka) \quad \text{for all } a\in A, \fraka \in A^{\ot k}.$$
Further we define a counit
\begin{equation}
\vep_\mo: \sha_\mo(A)\to \bfk, \quad \vep_\mo(\fraka):=(-\lambda)^k\vep( a_0 a_1\cdots a_k)\quad \text{if } \fraka= a_0\ot a_1\ot\cdots\ot a_k\in A^{\ot (k+1)},
\mlabel{eq:counit}
\end{equation}
for the counit $\vep:A\to \bfk$ of the bialgebra $A$. Then the definitions of $\vep_\mo$ and $P_A$ lead to
\begin{equation}
\vep_\mo (1\ot \fraka')=\vep_\mo(P_A(\fraka')) =-\lambda \vep_\mo(\fraka'), \quad \vep_\mo(a_0\ot \fraka')=-\lambda \vep(a_0)\vep_\mo(\fraka') \quad \text{for all } \fraka=a_0\ot \fraka'.
\mlabel{eq:veprec}
\end{equation}

\begin{lemma}
Let  $P_A$ and $\Delta_\mo$ be as above. Then for $\alpha\in A\barot A$, we have
$$ (\id\barot \Delta_\mo )(\id\barot P_A  )(\alpha)=((\id \barot  P_A  )(\alpha))\barot 1
+(\id\barot \id\barot P_A)(\id \barot \Delta_\mo)(\alpha) +\lambda \alpha \barot 1,$$
as an element in $A\barot A\barot A$.
\mlabel{lem:coc2}
\end{lemma}
\begin{proof}
We only need to verify the equation for a pure tensor $\alpha_0\barot \alpha_1\in A\barot A.$ For this we check
\begin{eqnarray*}
(\id\barot \Delta_\mo )(\id\barot P_A  )(\alpha_0\barot \alpha_1) &=&
\alpha_0 \barot (\Delta_\mo  P_A(\alpha_1))\\
&=& \alpha_0 \barot P_A(\alpha_1)\barot 1
+ \alpha_0 \barot \Big((\id \barot P_A)\Delta_\mo(\alpha_1)\Big)
+ \lambda \alpha_0\barot \alpha_1 \barot 1\\
&=& \big((\id\barot P_A)(\alpha)\big)\barot 1 + (\id \barot \id \barot P_A)(\id \barot \Delta_\mo)(\alpha) + \lambda \alpha \barot 1,
\end{eqnarray*}
as needed.
\end{proof}

\begin{lemma}
Let  $P_A $ be as above. Then
$ \id \barot P_A$ is a modified Rota-Baxter operator on the tensor product algebra $\sha_\mo(A) \oot \sha_\mo(A).$
\end{lemma}
\begin{proof} It follows directly from the definition of  the modified Rota-Baxter operator.
\end{proof}

\begin{theorem}
If $A$ is a bialgebra, then $\sha_\mo(A)$ is a bialgebra.
\mlabel{thm:bialg}
\end{theorem}
\begin{proof}
The proof is divided into the following four parts:
\begin{enumerate}
\item
The coproduct $\Delta_\mo:\sha_\mo(A)\to \sha_\mo(A)\oot \sha_\mo(A)$ is an algebra homomorphism;
\mlabel{it:coprodhom}
\item
The counit $\vep_\mo: \sha_\mo(A)\to \bfk$ is an algebra homomorphism;
\mlabel{it:counithom}
\item
The coproduct $\Delta_\mo$ is coassociative; and
\mlabel{it:coasso}
\item
The map $\vep_\mo$ satisfies the counicity property.
\mlabel{it:counit}
\end{enumerate}

\noindent
(\mref{it:coprodhom})
To prove that $\Delta_\mo: \sha_\mo(A)\to \sha_\mo(A)\oot \sha_\mo(A)$ is an algebra homomorphism, we verify the multiplicativity
\begin{equation}
 \Delta_\mo(\fraka \dm \frakb)= \Delta_\mo(\fraka) \dm \Delta_\mo(\frakb)
 \mlabel{eq:mult}
 \end{equation}
for any pure tensor $\fraka= a_0 \ot a_1 \ot \cdots \ot  a_m \in A^{\ot (m+1)} ,b= b_0 \ot  b_1 \ot \cdots \ot  b_n \in  A^{\ot (n+1)} .$
For this we apply the induction on $m+n\geq 0$. When $m=0$ and $n=0,$ this is simply the multiplicity of $\Delta$ on $A$.
For the inductive step, when $m+n\geq 1$ with either $m=0$ or $n=0$, Eq.~(\mref{eq:mult}) again follows the definition of $\dm$. When $m, n\geq 1$, write $\fraka=a_0\ot \fraka'$ and $\frakb=b_0\ot\frakb'$.

First assume that $a_0=b_0=1$. Then we have
\begin{eqnarray*}
\lefteqn{ \Delta_\mo(\fraka \dm \frakb)= \Delta_\mo(P_A(\fraka')\dm P_A(\frakb'))}\\
&=& \Delta_\mo\Big( P_A\big(\fraka'\dm P_A(\frakb')\big)\Big) + \Delta_\mo\Big(P_A\big(P_A(\fraka')\dm \frakb'\big)\Big)-\lambda^2 \Delta_\mo( \fraka'\dm \frakb')\quad \text{(by Eq.~(\mref{eq:mrbo}))}\\
&=& P_A\big(P_A(\fraka')\dm \frakb'\big)\oot 1 + (\id\oot P_A)\Delta_\mo\big(P_A(\fraka')\dm \frakb'\big) + \lambda \big(P_A(\fraka')\dm \frakb'\big)\oot 1 \\
&& + P_A\big(\fraka'\dm P_A(\frakb')\big)\oot 1 + (\id\oot P_A)\Delta_\mo\big(\fraka'\dm P_A(\frakb')\big) + \lambda \big(\fraka'\dm P_A(\frakb')\big)\oot 1 \\
&&-\lambda^2 \Delta_\mo( \fraka'\dm \frakb') \quad \text{(by Eq.~(\mref{eq:cocyc}))}\\
&=& P_A\big(P_A(\fraka')\dm \frakb'\big)\oot 1 + (\id\oot P_A)\Big(\Delta_\mo\big(P_A(\fraka')\big)\dm \Delta_\mo(\frakb')\Big)
+ \lambda \big(P_A(\fraka')\dm \frakb'\big)\oot 1 \\
&& + P_A\big(\fraka'\dm P_A(\frakb')\big)\oot 1 + (\id\oot P_A)\Big(\Delta_\mo(\fraka')\dm \Delta_\mo\big(P_A(\frakb')\big)\Big) + \lambda \big(\fraka'\dm P_A(\frakb')\big)\oot 1 \\
&&-\lambda^2 \Delta_\mo( \fraka')\dm \Delta_\mo(\frakb') \quad \text{(by the induction hypothesis)}\\
&=& P_A\big(P_A(\fraka')\dm \frakb'\big)\oot 1
+ (\id\oot P_A)\Big(\Big(P_A(\fraka')\oot 1
+(\id\ot P_A) \Delta_\mo(\fraka')
\lambda (\fraka'\oot 1)\Big)\dm \Delta_\mo(\frakb')\Big)\\
&& + \lambda \big(P_A(\fraka')\dm \frakb'\big)\oot 1
+ P_A\big(\fraka'\dm P_A(\frakb')\big)\oot 1 \\
&&+ (\id\oot P_A)\Big(\Delta_\mo(\fraka')\dm \Big(P_A(\frakb')\oot 1+ (\id\ot P_A) \Delta_\mo(\frakb')+\lambda \frakb'\oot 1\Big)\Big) + \lambda \big(\fraka'\dm P_A(\frakb')\big)\oot 1 \\
&&-\lambda^2 \Delta_\mo( \fraka')\dm \Delta_\mo(\frakb') \quad \text{(by Eq.~(\mref{eq:cocyc}))}\\
&=& P_A\big(P_A(\fraka')\dm \frakb'\big)\oot 1
+ (L_{P_A(\fraka')}\oot P_A)\Delta_\mo(\frakb')+ (\id\oot P_A)\Big((\id\ot P_A) (\Delta_\mo(\fraka'))\dm \Delta_\mo(\frakb')\Big) \\
&&+ \lambda (L_{\fraka'} \oot P_A) \Delta_\mo(\frakb')
+ \lambda \big(P_A(\fraka')\dm \frakb'\big)\oot 1
+ P_A\big(\fraka'\dm P_A(\frakb')\big)\oot 1 \\
&& + (L_{P_A(\frakb')} \oot P_A) (\Delta_\mo(\fraka')) +
(\id\oot P_A)\Big(\Delta_\mo(\fraka')\dm (\id\ot P_A) \Delta_\mo(\frakb')\Big)
+\lambda (L_{\frakb'}\oot P_A)(\fraka') \\
&&+ \lambda \big(\fraka'\dm P_A(\frakb')\big)\oot 1
-\lambda^2 \Delta_\mo( \fraka')\dm \Delta_\mo(\frakb') \quad \text{(write } L_u \text{ for multiplication by } u)
\end{eqnarray*}

On the other hand,
\begin{eqnarray*}
\lefteqn{\Delta_\mo(\fraka)\dm\Delta_\mo(\frakb) = \Delta_\mo(P_A(\fraka')\dm P_A(\frakb')) }\\
&=& \Big(P_A(\fraka')\oot 1 + (\id\oot P_A)\Delta_\mo(\fraka')+\lambda \fraka'\oot 1\Big)\dm \Big(P_A(\frakb')\oot 1 + (\id\oot P_A)\Delta_\mo(\frakb')+\lambda \frakb'\oot 1\Big)\\
&=& (P_A(\fraka')\dm P_A(\frakb'))\oot 1
+ (P_A(\fraka')\oot 1)\dm \big((\id\oot P_A)\Delta_\mo(\frakb')\big)
+\big(P_A(\fraka')\oot 1)\dm(\lambda \frakb'\oot 1) \\
&&+\big((\id\oot P_A)\Delta_\mo(\fraka')\big)\dm (P_A(\frakb')\oot 1) +\big((\id\oot P_A)\Delta_\mo(\fraka')\big)\dm \big((\id\oot P_A)\Delta_\mo(\frakb')\big) \\
&&+\big((\id\oot P_A)\Delta_\mo(\fraka')\big)\dm (\lambda \frakb' \oot 1) + (\lambda\fraka'\oot 1)\dm (P_A(\frakb')\oot 1)\\
&&+ (\lambda\fraka'\oot 1)\dm \big((\id\ot P_A)\Delta_\mo(\frakb')\big)
+\lambda^2\fraka'\frakb'\oot 1 \quad \text{(distribute)}\\
&=& P_A\big(P_A(\fraka')\dm \frakb'\big)\oot 1
+ P_A\big(\fraka' \dm P_A(\frakb')\big)\oot 1
-\lambda^2\fraka'\frakb'\oot 1
+ (P_A(\fraka')\oot 1)\dm \big((\id\oot P_A)\Delta_\mo(\frakb')\big)\\
&&+\big(P_A(\fraka')\oot 1)\dm(\lambda \frakb'\oot 1)
+\big((\id\oot P_A)\Delta_\mo(\fraka')\big)\dm (P_A(\frakb')\oot 1) \\ &&+\,\big((\id\oot P_A)\Delta_\mo(\fraka')\big)\dm \big((\id\oot P_A)\Delta_\mo(\frakb')\big)
+\big((\id\oot P_A)\Delta_\mo(\fraka')\big)\dm (\lambda \frakb' \oot 1) \\
&& + (\lambda\fraka'\oot 1)\dm (P_A(\frakb')\oot 1) + (\lambda\fraka'\oot 1)\dm \big((\id\ot P_A)\Delta_\mo(\frakb')\big)
+\lambda^2\fraka'\frakb'\oot 1 \quad \text{(by Eq.~(\mref{eq:mrbo}))}\\
&=& P_A\big(P_A(\fraka')\dm \frakb'\big)\oot 1
+ P_A\big(\fraka' \dm P_A(\frakb')\big)\oot 1
+ (L_{P_A(\fraka')}\oot P_A)\Delta_\mo (\frakb') \\
&&+ \lambda (P_A(\fraka')\dm \frakb')\oot 1
+ (L_{P_A(\frakb')}\oot P_A)\Delta_\mo (\fraka')
+ (\id\oot P_A)\Big((\id\ot P_A)\big(\Delta_\mo(\fraka')\big)\dm \Delta_\mo(\frakb')\Big) \\
&&
+ (\id\oot P_A)\Big(\Delta_\mo(\fraka')\dm\big((\id\ot P_A)\Delta_\mo(\frakb')\big)\Big)
-\lambda^2 \Delta_\mo(\fraka')\dm\Delta_\mo(\frakb')
+ \lambda(L_{\frakb'}\oot P_A)\Delta_\mo(\fraka') \\
&& + \lambda \big(\fraka'\dm P_A(\frakb')\big) \oot 1 +\lambda (L_{\fraka'}\oot P_A)\Delta_\mo(\frakb') \quad \text{(write } L_u \text{ for multiplication by } u).
\end{eqnarray*}

Now we see that the $i$-th term in the expansion of $ \Delta_\mo(\fraka)\diamond_\mo \Delta_\mo( \frakb) $ agrees with the $\sigma(i)$-th term in the expansion of $\Delta_\mo(\fraka \diamond_\mo \frakb)$, where $\sigma \in \Sigma_{11}$ is
$$\left (\begin{array}{ll}
\quad i \\
\sigma (i)
\end{array}  \right ) =\left (\begin{array}{ccccccccccc}
1&  2&  3&  4&  5&  6&  7&  8&  9 &  10&  11 \\
1&  3&  6&  11&  4&  2&  5&  7&  9&  10&  8
\end{array}  \right ). $$

Next, for general $\fraka=a_0\ot \fraka'$ and $\frakb=b_0\ot \frakb'$, by the case considered above, we have

\begin{eqnarray*}
\Delta_\mo(\fraka)\dm \Delta_\mo(\frakb)&=&
\Delta(a_0)\dm \Delta_\mo(1\ot \fraka')\dm \Delta(b_0)\dm \Delta_\mo(1\ot \frakb')\\
&=&
\Delta(a_0)\dm \Delta(b_0)\dm \Delta_\mo(P_A(\fraka')\dm\Delta_\mo(P_A(\frakb')) \\
&=&
\Delta(a_0b_0)\dm \Delta_\mo\big(P_A(\fraka')\dm P_A(\frakb')\big).
\end{eqnarray*}

On the other hand, by the definition of $\Delta_\mo$ and $\dm$, we have
\begin{eqnarray*}
&&\Delta_\mo(\fraka\dm \frakb)\\
&=&\Delta_\mo\Big((a_0b_0) \ot \big((1\ot\fraka')\dm \frakb'\big)\Big)
+\Delta_\mo \Big((a_0b_0)\ot\big(\fraka'\dm (1\ot \frakb')\big)\Big)\\
&&+\Delta_\mo \Big((a_0b_0)(\fraka'\dm \frakb')\Big) \quad \text{(by Eq.~(\mref{eq:mrbprod}))}\\
&=& \Delta(a_0b_0)\dm \Delta_\mo\Big(1\ot\big(\fraka'\dm (1\ot \frakb')\big)\Big)+\Delta(a_0b_0)\dm \Delta_\mo \Big(1\ot\big(\fraka'\dm (1\ot \frakb')\big)\Big)\\
&&+\Delta(a_0b_0)\dm\Delta_\mo \Big(\fraka'\dm \frakb'\Big).
\end{eqnarray*}
Since $P_A$ is a modified Rota-Baxter operator, this agrees with the expansion of $\Delta_\mo(\fraka)\dm \Delta_\mo(\frakb)$.

This completes the verification that $\Delta_\mo$ is an algebra homomorphism for the product $\diamond_\mo$.
\smallskip

\noindent
(\mref{it:counithom})
We also prove the multiplicativity of $\vep_\mo$:
$$\vep_\mo(\fraka \dm \frakb)=\vep_\mo(\fraka)\dm\vep_\mo(\frakb)$$
by induction on $m+n\geq 0$ for $\fraka\in A^{\ot (m+1)}, \frakb\in A^{\ot (n+1)}, m, n\geq 0$. The initial case of $m+n=0$ is simply the multiplicativity of the counit $\vep$ on $A$. Assume that the equation holds for  $m+n\geq k\geq 0$ and consider the case when $m+n=k+1$. If either $m=0$ or $n=0$, then the equation follows easily from the definitions of $\dm$ and $\vep_\mo$. So we consider the case of $m, n\geq 1$ and write $a= a_0 \ot  a_1 \ot \cdots \ot a_m= a_0\ot\fraka'$ and $\frakb= b_0 \ot  b_1 \ot \cdots \ot b_n= b_0\ot\frakb'.$
First assume $a_0=b_0=1$. Then
\begin{eqnarray*}
&&\vep_\mo \big( (1\ot \fraka')\dm (1\ot \frakb')\big)\\
&=&
\vep_\mo \big( P_A(\fraka')\dm P_A(\frakb')\big) \\
&=& \vep_\mo\big(P_A(P_A(\fraka')\dm \frakb')\big)
+\vep_\mo\big(P_A(\fraka'\dm P_A(\frakb')\big)
-\lambda^2 \vep_\mo(\fraka'\dm\frakb') \quad \text{(by Eq.~(\mref{eq:mrbo}))}\\
&=& -\lambda \vep_\mo\big(P_A(\fraka')\dm\frakb'\big)
-\lambda\vep_\mo\big(\fraka'\dm P_A(\frakb')\big)
-\lambda^2 \vep_\mo(\fraka'\dm\frakb')
\quad \text{(by Eq.~(\mref{eq:veprec}))}\\
&=&-\lambda \vep_\mo\big(P_A(\fraka')\big)\dm\vep_\mo\big(\frakb'\big)
-\lambda\vep_\mo\big(\fraka'\big)\dm \vep_\mo\big(P_A(\frakb')\big)
-\lambda^2 \vep_\mo(\fraka')\dm\vep_\mo(\frakb') \\
&&
\quad \text{(by the induction hypothesis)} \\
&=&\lambda^2 \vep_\mo(\fraka')\dm\vep_\mo(\frakb')
+\lambda^2\vep_\mo(\fraka')\dm \vep_\mo(\frakb')
-\lambda^2 \vep_\mo(\fraka')\dm\vep_\mo(\frakb')
\quad \text{(by Eq.~(\mref{eq:veprec}))} \\
&=& \lambda^2 \vep_\mo(\fraka')\dm\vep_\mo(\frakb')\\
&=& \vep_\mo (1\ot \fraka')\dm \vep_\mo(1\ot \frakb').
\end{eqnarray*}
In general,
\begin{eqnarray*}
&&\vep_\mo(\fraka \dm \frakb)\\
&=&
\vep_\mo\big((a_0b_0)\dm P_A(\fraka')\dm P_A(\frakb)\big) \quad \text{(by Eq.~(\mref{eq:mrbprod}))} \\
&=& \vep(a_0b_0)\dm \vep_\mo\big(P_A(\fraka')\dm P_A(\frakb')\big)
\quad \text{(by Eq.~(\mref{eq:veprec}))}\\
&=& \vep(a_0)\dm \vep(b_0)\dm \vep_\mo(1\ot \fraka')\dm \vep_\mo(1\ot \frakb') \quad \text{(by the previous case)}\\
&=& \vep_\mo(\fraka) \dm \vep_\mo(\frakb).
\end{eqnarray*}
This proves the multiplicativity of $\vep_\mo$.
\smallskip

\noindent
(\mref{it:coasso})
We will prove the coassociativity of $\Delta_\mo$ by induction on $k\geq 0$ for $\fraka\in A^{\ot (k+1)}.$
When $k=0,$ then  $\Delta_\mo(\fraka)=\Delta(\fraka) $ and we are done. For $k>0$, write $\fraka= a_0 \ot \fraka' \in A^{\ot (k+1)} $ with  $\fraka' \in A^{\ot k}.$ First taking $a_0=1$, then $\fraka=P_A(\fraka')$ and we have
\begin{eqnarray*}
&&(\id\barot\Delta_\mo)\Delta_\mo(\fraka)\\
&=& (\id\barot\Delta_\mo)\Delta_\mo\big(P_A(\fraka')\big)\\
&=& (\id\barot\Delta_\mo)\Big( P(\fraka')\oot 1+(\id\oot P_A)\Delta_\mo(\fraka') + \lambda \fraka'\oot 1\Big) \quad \text{(by Eq.~(\mref{eq:cocyc}))}\\
&=& P_A(\fraka')\oot 1\oot 1 + \big(\id \oot (\Delta_\mo P_A)\big)\Delta(\fraka') +\lambda \fraka'\oot 1\oot 1\\
&=& P_A(\fraka')\oot 1\oot 1 + \big((\id\oot P_A)\Delta_\mo(\fraka')\big)\oot 1 + (\id \oot \id \oot P_A)(\id \oot \Delta_\mo)\Delta_\mo(\fraka') \\
&&+\lambda \Delta_\mo(\fraka')\oot 1 +\lambda \fraka'\oot 1\oot 1 \quad \text{(by Lemma~\mref{lem:coc2})}\\
&=& P_A(\fraka')\oot 1\oot 1 + \big((\id\oot P_A)\Delta_\mo(\fraka')\big)\oot 1 + (\id \oot \id \oot P_A)(\Delta_\mo \oot \id)\Delta_\mo(\fraka') \\
&&+\lambda \Delta_\mo(\fraka')\oot 1 +\lambda \fraka'\oot 1\oot 1 \quad \text{(by the induction hypothesis)}.
\end{eqnarray*}

On the other hand,
\begin{eqnarray*}
&& (\Delta_\mo\oot \id)\Delta_\mo(\fraka)\\
&=& (\Delta_\mo\oot \id)\Delta_\mo(P_A(\fraka'))\\
&=& (\Delta_\mo \oot \id)\Big( P(\fraka')\oot 1+(\id\oot P_A)\Delta_\mo(\fraka') + \lambda \fraka'\oot 1\Big) \quad \text{(by Eq.~(\mref{eq:cocyc}))}\\
&=& (\Delta_\mo P_A(\fraka'))\oot 1 + (\Delta_\mo \oot P_A)\Delta_\mo(\fraka') + \lambda \Delta_\mo(\fraka')\oot 1\\
&=& P_A(\fraka')\oot 1\oot 1 + \big((\id\oot P_a)\Delta_\mo(\fraka')\big) \oot 1 +\lambda \fraka'\oot 1\oot 1
+ (\id\oot \id \oot P_A)(\Delta_\mo\oot \id)\Delta_\mo(\fraka') +\lambda \Delta(\fraka')\oot 1
\end{eqnarray*}
by Eq.~(\mref{eq:cocyc}).
This agrees with the expansion of $(\id\oot \Delta_\mo)\Delta_\mo(\fraka)$. For the general case of $\fraka=a_0\ot \fraka'=a_0\dm P_A(\fraka')$, we have
\begin{eqnarray*}
 (\Delta_\mo\oot \id)\Delta_\mo (a_0P_A(\fraka'))
&=& (\Delta_\mo\oot \id)\Big(\Delta(a_0)\dm \Delta_\mo(P_A(\fraka'))\Big)\\
&=& \Big((\Delta\oot \id)\Delta(a_0)\Big)\dm \Big((\Delta_\mo\oot \id)\Delta_\mo(P_A(\fraka'))\Big).
\end{eqnarray*}
By the coassociativity of $\Delta$ and the previous case of the proof, this agrees with
\begin{eqnarray*}
 (\id \oot \Delta_\mo)\Delta_\mo (a_0P_A(\fraka'))
&=& (\id \oot \Delta_\mo)\Big(\Delta(a_0)\dm \Delta_\mo(P_A(\fraka'))\Big)\\
&=& \Big((\id \oot \Delta)\Delta(a_0)\Big)\dm \Big((\id \oot \Delta_\mo)\Delta_\mo(P_A(\fraka'))\Big),
\end{eqnarray*}
as needed.
This completes the induction.
\smallskip

\noindent
(\mref{it:counit})
For the counicity property of $\vep_\mo$, we verify
$$
(\id\barot \vep_\mo )\Delta_\mo(\fraka)=a\barot 1, \quad
(\vep_\mo \barot \id )\Delta_\mo(\fraka)=1\barot a
$$
by induction on $m\geq 0$ for $\fraka\in A^{\ot(m+1)}.$

When $m=0,$   $\vep_\mo=\vep $, so we are done by the counicity property of $\vep$.

For $m>0$, $\fraka= a_0 \ot \fraka' \in A^{\ot (m+1)} $ with  $\fraka' \in A^{\ot m}.$ We have
\begin{eqnarray*}
\lefteqn{(\id\barot \vep_\mo )\Delta_\mo(\fraka)=(\id\barot \vep_\mo) \Delta_\mo ( a_0 \diamond_\mo P_A (\fraka')) }\\
&=&(\id\barot \vep_\mo) \Delta_\mo ( a_0) \diamond_\mo (\id\barot \vep_\mo)( \Delta_\mo(P_A (\fraka')) \\
&=&( a_0\barot 1) \diamond_\mo (\id\barot \vep_\mo) (P_A  (\fraka')\barot 1+(\id\barot P_A )\Delta_\mo(\fraka')+\lambda \fraka'\barot 1) \\
&=&( a_0\barot 1) \diamond_\mo  (P_A  (\fraka')\barot 1+(\id\barot \vep_\mo P_A )\Delta_\mo(\fraka')+(\id\barot \vep_\mo)(\lambda \fraka'\barot 1))  \\
&=&( a_0\barot 1) \diamond_\mo  (P_A  (\fraka')\barot 1-\lambda(\id\barot \vep_\mo)\Delta_\mo(\fraka')+(\id\barot \vep_\mo)(\lambda \fraka'\barot 1))  \\
&=&( a_0\barot 1) \diamond_\mo (P_A  (\fraka')\barot 1-\lambda \fraka'\barot 1+\lambda \fraka'\barot 1)=\fraka \barot 1.
\end{eqnarray*}
Further,
  \begin{eqnarray*}
\lefteqn{(\vep_\mo \barot \id )\Delta_\mo(\fraka)=(\vep_\mo \barot \id) \Delta_\mo ( a_0 \diamond_\mo P_A (\fraka')) }\\
&=&(\vep_\mo \barot \id) \Delta_\mo ( a_0) \diamond_\mo (\vep_\mo \barot \id) \Delta_\mo(P_A (\fraka')) \\
&=&(1\barot  a_0) \diamond_\mo (\vep_\mo \barot \id) (P_A  (\fraka')\barot 1+(\id\barot P_A )\Delta_\mo(\fraka')+\lambda \fraka'\barot 1) \\
&=&(1\barot  a_0) \diamond_\mo  (\vep_\mo P_A (\fraka')\barot 1+(\vep_\mo \barot  P_A )\Delta_\mo(\fraka')+\lambda \vep_\mo(\fraka')\barot 1)\\
&=&(1\barot  a_0) \diamond_\mo  (-\lambda \vep_\mo(\fraka')\barot 1+(\id \barot  P_A )(\vep_\mo \barot \id)\Delta_\mo(\fraka')+\lambda \vep_\mo(\fraka')\barot 1)\\
&=&(1\barot  a_0) \diamond_\mo (\id \barot  P_A )(1\barot \fraka'))=(1\barot  a_0)\diamond_\mo (1\barot P_A  (\fraka') )=1 \barot \fraka.
\end{eqnarray*}

This concludes the verification of the four properties for $\sha_\mo(A)$ to be a bialgebra.
\end{proof}

\subsection{The Hopf algebra structure}
\mlabel{ss:hopf}

\begin{defn}
A {\bf filtered bialgebra} is a bialgebra $(A, m, \mu, \Delta, \vep)$ together with an increasing filtration $A_n, n\geq 0,$ such that
$$A=\cup_{n=0}^\infty A_n, \quad A_mA_n\subseteq A_{m+n}, \quad \Delta(A_k)\subseteq \sum_{m+n=k}A_m\ot A_n, \quad m, n, k\geq 0.$$
A filtered bialgebra $A$ is called {\bf connected} if  $A_0={\bf k}$.
\end{defn}
The following theorem is from ~\mcite{DM}.
\begin{prop}
A connected filtered bialgebra is a Hopf algebra.
\mlabel{pp:conn}
\end{prop}

\begin{theorem}
If $A=\cup_{n\geq 0}A_n$ is a connected filtered bialgebra, then $\sha_\mo (A)$ is also a connected filtered bialgebra. In particular, $\sha_\mo (A)$ is a Hopf algebra.
\mlabel{thm:hopf}
\end{theorem}
\begin{proof}
For $0\neq a \in A$, we define the degree of $a$ by
$$ \deg(a)=\min\{ k\,|\, a\in A_k\}.$$
In other words, $\deg(a)=k$ for $a\in A_k\setminus A_{k-1}, k\geq 0\ (\text{we set}\  A_{-1}=0).$
Denote $\frakA:=\sha_\mo(A)$.
For  $0\neq \fraka=a_0\ot \cdots \ot a_m\in A^{\ot (m+1)}\subseteq \sha_\mo (A)$, we define
$$ \deg(\fraka):= \deg(a_0)+\cdots + \deg(a_m)+m$$
and let
$\mathfrak{A}_k$ denote the linear span of pure tensors $\fraka\in\frakA$ with $\deg(\fraka)\leq k$.
Then
$\mathfrak{A}_0=A_0={\bf k}$ and, for $\fraka=a_0\ot \fraka'\in A^{\ot (m+1)}$ with $\fraka'\in A^{\ot m}$, we have
\begin{equation}
\deg(\fraka)=\deg(a_0)+\deg(\fraka')+1, \quad \fraka'\in \frakA_n \Longrightarrow \fraka \in \frakA_{n+\deg(a_0)+1}.
\mlabel{eq:deg}
\end{equation}

We will show that $\mathfrak{A}$ is filtered. Then by $\frakA_0=\bfk$, $\frakA$ is connected which completes the proof of the theorem.

First we apply the induction on $m+n$ to prove
$$\fraka\diamond_\mo \frakb\in\mathfrak{A}_{m+n}$$
for pure tensors $\fraka\in \mathfrak{A}_m$ and $\frakb\in\mathfrak{A}_n$. If $m+n=0,$ then $m=n=0$. Then the equation holds since $\mathfrak{A}_0={\bfk}.$ For a given $k\geq 0$, assume that the equation holds for $m+n=k$ and consider the case when $m+n=k+1$. If $m=0$ or $n=0$, then $\fraka\in \bfk$ or $\frakb\in \bfk$ and so the inclusion clearly holds. Thus we assume $m, n \geq 1.$ If $\fraka\in A$ or $\frakb\in A$, then the equation holds since $A$ is filtered. So we can further take $\fraka\in A^{\ot p}, \frakb \in A^{\ot q}$ with $p, q\geq 2$. Then we can write $\fraka= a_0 \ot \fraka', \frakb= b_0 \ot \frakb'$ with $\fraka'= a_1 \ot \cdots \ot  a_p, \frakb'= b_1 \ot  \cdots \ot  b_q.$
By Eq.~(\mref{eq:mrbprod}), we have
$$\fraka\diamond_\mo \frakb=( a_0 \ot \fraka')\diamond_\mo ( b_0 \ot \frakb')= a_0  b_0 \ot \big((1 \ot \fraka')\diamond_\mo  \frakb'\big) + a_0  b_0 \ot \big(\fraka'\diamond_\mo (1 \ot \frakb')\big) -\lambda^2  a_0   b_0 (\fraka'\diamond_\mo \frakb').$$
Consider the first term on the right hand side. The first equation in Eq.~(\mref{eq:deg}) gives
$$\deg(\fraka') = \deg(\fraka) - \deg(a_0) - 1 \leq m-\deg(a_0) -1\,\text{ and }\,\deg(\frakb') = \deg(\frakb) - \deg(b_0) - 1 \leq n-\deg(a_0) -1,$$
whence
$$\fraka'\in \frakA_{m-\deg(a_0) -1},\, 1\ot \fraka'\in \frakA_{m-\deg(a_0)}\,\text{ and }\,  \frakb'\in \frakA_{n-\deg(b_0) -1}.$$
Thus by the induction hypothesis and the second equation in Eq.~(\mref{eq:deg}), we obtain
$$(1\ot \fraka')\diamond_\mo \frakb'\in \frakA_{m+n-\deg(a_0)-\deg(b_0)-1}\,\text{ and }\,
 a_0 b_0 \ot \big((1 \ot \fraka')\diamond_\mo  \frakb'\big) \in \frakA_{m+n}.$$
The same argument applies to the second and the third terms. Thus we obtain $\fraka\diamond_\mo \frakb\in\mathfrak{A}_{m+n},$ completing the induction.

Second, we also use the induction on $k\geq 0$ to show
$$\Delta_\mo (\mathfrak{A}_k)\subseteq \sum_{m+n=k}\mathfrak{A}_m\oot \mathfrak{A}_n.$$
It is true for $k=0$ since $\mathfrak{A}_0={\bfk}\subseteq A$ and $A$ is filtered. Assume that the equation holds for a given $k\geq 0$ and consider $\frakA_{k+1}$. For $\fraka\in \frakA_{k+1}$, if $\fraka\in A$, then there is nothing to prove since $A$ is filtered. So we can take $\fraka\in A^{\ot i}$ with $i\geq 2$ and write $\fraka=a_0\ot \fraka'$.  Then the multiplicativity of $\Delta_\mo$ and the cocycle condition in Eq.~(\mref{eq:cocyc}) lead to
  \begin{eqnarray*}
\Delta_\mo (\fraka)&=&\Delta_\mo( a_0 \diamond_\mo P_A  (\fraka'))\\
  &=&\Delta ( a_0)\diamond_\mo \Delta_\mo(P_A  (\fraka'))\\
  &=&\Delta ( a_0)\diamond_\mo \big(P_A  (\fraka')\barot 1+(\id\barot P_A  )\Delta_\mo(\fraka')+ \lambda \fraka'\barot 1\big).
\end{eqnarray*}
By the induction hypothesis, we have
$$\Delta_\mo(\fraka')\subseteq \sum_{m+n=\deg(\fraka')}\frakA_m \oot \frakA_n.$$
Then by Eq.~(\mref{eq:deg}) and $\sum_{m+n=\ell}\frakA_m\oot \frakA_n \subseteq \sum_{m+n=\ell+1}\frakA_m\oot \frakA_n$, we obtain
$$ P_A  (\fraka')\barot 1+(\id\barot P_A  )\Delta_\mo(\fraka')+ \lambda \fraka'\barot 1 \in \sum_{m+n=\deg(P_A(\fraka'))} \frakA_m\oot \frakA_n.$$
Then by Eq.~(\mref{eq:deg}) again, we obtain
$$ \Delta ( a_0)\diamond_\mo \big(P_A  (\fraka')\barot 1+(\id\barot P_A  )\Delta_\mo(\fraka')+ \lambda \fraka'\barot 1\big) \in \sum_{m+n=\deg(a_0)\atop+\deg(P_A(\fraka'))} \frakA_m\oot \frakA_n
=\sum_{m+n=k+1}\frakA_m\oot \frakA_n.$$
This completes the induction.
\end{proof}

\subsection{The case of $A=\bfk$ revisited}
\mlabel{ss:case2}
We end the paper by returning to the special case of $\sha_\mo(A)$ when $A$ is the base ring $\bfk$, first considered in Section~\mref{ss:case}. As what we achieved for the product $\diamond_\mo$ of $\sha_\mo(\bfk)$ in Proposition~\mref{pp:mult}, we will give an explicit formula for the coproduct of $\sha_\mo(\bfk)$.

Note that $\bfk$ is naturally a connected $\bfk$-bialgebra with its coproduct $$\Delta: \bfk\rightarrow \bfk \ot \bfk,  \quad  x\mapsto x\ot {\bf 1}\quad \text {for all } x\in \bfk,$$
counit
$$\vep=id_{\bfk}: \bfk\rightarrow \bfk$$
and filtration
$$ \bfk_k=\bfk \quad \text{for all } k\geq 0.$$

As in Section~\mref{ss:case}, $\sha_\mo(\bfk)=\bigoplus_{k\geq 0} {\bfk}^{\ot (k+1)}=\bigoplus_{k\geq 0} \bfk u_k$. The coproduct $\Delta_\mo$ in Section~\mref{ss:bialg} becomes
\begin{align*}
\Delta_\mo:& \sha_\mo(\bfk) \rightarrow \sha_\mo(\bfk) \oot \sha_\mo(\bfk), \\
&u_k \mapsto \left\{\begin{array}{ll}
 u_0 \oot u_0, & k=0, \\
P_\bfk  (u_{n-1})\barot 1+(\id\barot P_\bfk)\Delta_\mo(u_{n-1})+ \lambda u_{n-1}\barot 1, & k\geq 1.
\end{array} \right .
\end{align*}
Then we have
\begin{prop}
$\Delta_\mo(u_n)=\sum\limits_{r=0}^n u_r \barot u_{n-r} +\lambda \sum\limits_{r=0}^{n-1} u_r \barot u_{n-r-1}\  \text {for} \ n\geq1.$
\end{prop}

\begin{proof}
We use induction on $n\geq 1$. The initial step of $n=1$ follows from
\begin{eqnarray*}
\Delta_\mo(u_1)&=& P_A  (u_0)\barot 1+(\id\barot P_A  )\Delta_\mo(u_0)+ \lambda u_0\barot 1\\
&=&u_1\barot 1+(\id\barot P_A  )(u_0\oot u_0)+ \lambda u_0\barot 1\\
&=&u_1\barot u_0 +u_0\oot u_1+ \lambda u_0\barot u_0.
\end{eqnarray*}
Now for the induction step of $n\geq 2,$ applying the induction hypothesis, we obtain
\begin{eqnarray*}
\Delta_\mo(u_n)&=&P_A  (u_{n-1})\barot 1+(\id\barot P_A  )\Delta_\mo(u_{n-1})+ \lambda u_{n-1}\barot 1\\
&=& u_n\barot 1+(\id\barot P_A  )\left(\sum\limits_{r=0}^{n-1} u_r \barot u_{n-1-r} +\lambda \sum\limits_{r=0}^{n-2} u_r \barot u_{n-2-r}\right) + \lambda u_{n-1}\barot 1\\
&=& u_n\barot 1+\sum\limits_{r=0}^{n-1} u_r \barot u_{n-r}+ \lambda \sum\limits_{r=0}^{n-2} u_r \barot u_{n-1-r}+\lambda u_{n-1}\barot 1\\
&=&\sum\limits_{r=0}^n u_r \barot u_{n-r} +\lambda \sum\limits_{r=0}^{n-1} u_r \barot u_{n-1-r}.
\end{eqnarray*}
Hence we complete the induction.
\end{proof}

Further, from Eq.~(\mref{eq:counit}), the counit is
$$\vep_\mo: \sha_\mo(\bfk)\rightarrow \bfk, u_n \mapsto (-\lambda)^n u_0.$$

\smallskip

\noindent {\bf Acknowledgements}: This work was supported by the National Natural Science Foundation of China (Grant No.~11771190).

\end{document}